\documentclass[11pt]{article}

\usepackage{amssymb}
\usepackage{amsmath}
\usepackage{mathrsfs}
\usepackage{amsfonts}
\usepackage{mathtools}
\usepackage{vmargin}
\usepackage{amsthm}
\usepackage{graphicx}
\usepackage[sort&compress, numbers]{natbib}
\usepackage{xfrac}
\usepackage{hyperref}
\usepackage{verbatim}
\usepackage[utf8]{inputenc}
\usepackage[T1]{fontenc}
\usepackage[english]{babel}

\numberwithin{equation}{section}

\setmarginsrb{20mm}{20mm}{20mm}{20mm}{10mm}{10mm}{10mm}{10mm}

\newtheorem{defi}{Definition}[section]
\newtheorem{rem}[defi]{Remark}
\theoremstyle{plain}

\newtheorem{tw}[defi]{Theorem}

\newtheorem{lem}[defi]{Lemma}

\newtheorem{prop}[defi]{Proposition}

\newcommand\bb[1]{\mathbf{#1}}

\renewcommand{\r}{\mathbb{R}}
\newcommand{\n}{\mathbb{N}}
\newcommand{\rn}{\mathbb{R}^d}
\newcommand{\tn}{\Omega}

\newcommand{\io}{\int_{\Omega}}

\newcommand{\htwo}{H^2(\Omega)}
\renewcommand{\hom}{H^1(\tn)}
\newcommand{\ld}{L^2(\Omega)}
\newcommand{\loc}{\operatorname{loc}}
\newcommand{\ul}{\frac{1}{2}}
\newcommand{\sul}{\sfrac{1}{2}}
\newcommand{\ddt}{\frac{d}{dt}}

\newcommand{\hp}{H^{\perp}}
\newcommand{\ant}{\quad\textrm{ and }\quad}
\renewcommand{\div}{\operatorname{div}}
\DeclareMathOperator{\curl}{curl}
\newcommand{\F}{\mathcal{F}}
\newcommand{\T}{\mathcal{T}}
\let \epsilon\varepsilon
\let \phi\varphi
\newcommand{\initcond}{\|\bb{v}_0\|^2_{\H}+(2\mu+\lambda)\|\nabla\div\bb u_0\|^2_{\ld}+\mu\|\curl\curl\bb u_0\|_{\ld}^2+\left\|\frac{\nabla\theta_0}{{\theta^{\sul}_0}}\right\|^2_{\ld}\leq D}
\newcommand{\M}{\mathcal{M}}
\newcommand{\ot}{(0,\infty)\times\Omega}
\newcommand{\wspnd}{2\mu+\lambda}
\newcommand{\bu}{\bb{u}}
\newcommand{\bs}[1]{\boldsymbol{#1}}
\newcommand{\N}{\mathcal{N}}
\renewcommand{\H}{\mathcal{H}}
\newcommand{\bg}{\bs{\gamma}}

\usepackage{xcolor}

\begin{document}
	\date{\today}
	\title{\bf Existence, uniqueness, and time-asymptotics of regular solutions in multidimensional thermoelasticity on domains with boundary}
	\author{Piotr Micha{\l} Bies\\
\it\small{Faculty of Mathematics and Information Sciences,}\\
\it\small{Warsaw University of Technology,}\\
\it\small{Ul. Koszykowa 75, 00-662 Warsaw, Poland.}\\
{\tt biesp@mini.pw.edu.pl}\\
{\small ORCID: 0000-0001-5385-1210}}
\maketitle
	
\begin{abstract}
\noindent In the paper, we investigate the nonlinear thermoelasticity model in two- and three-dimensional convex and bounded domains. We propose new boundary conditions for the displacement. These conditions are not usual in thermoelasticity. Whereas, we posit the Neumann boundary condition for the temperature. We prove the existence of global, unique solutions for small initial data. The temperature positivity is also shown. Next, we investigate the long-time behavior of solutions. We show that the divergence-free part of the displacement oscillates. On the other hand, we prove that the potential part and the temperature are strongly coupled. The non-rotation part is heavily affected by heat propagation. It turns out that it tends to $0$ as $x$ approaches infinity. Additionally, the temperature converges to a constant function.

Our techniques are firmly based on the functional $\F$ adopted from  {\sc Bies, P. M., Cie\'slak, T., Fuest, M., Lankeit, J., Muha, B., and Trifunovi\'c, S.}, \emph{Existence, uniqueness, and long-time asymptotic behavior of regular solutions in multidimensional thermoelasticity}, arXiv: 2507.20794, 2025. The functional is based on the Fisher information and higher-order derivatives of the displacement. It responds well to the new boundary conditions. It allows us to close a priori estimates. We also need $L^{\infty}$-estimates for the temperature here. The Moser iterative procedure ensures it.

The Helmholtz decomposition is applied to the displacement. The boundary conditions are crucial here. The boundary integrals that appear in the calculations at this point disappear thanks to these conditions. This allows us to split the problem into two separate ones. Each of them is associated with one part of the Helmholtz decomposition.
\end{abstract}

\medskip

\medskip

\medskip

\section{Introduction}

Let us take $\Omega\subset\rn$. We assume that $\Omega$ is open, bounded, and convex with a smooth boundary $\partial\Omega$. We want to investigate the following problem
\begin{equation}\label{system}
\begin{cases}
\bb{u}_{tt}- (2\mu+\lambda)\nabla\div\bb u+\mu\curl\curl\bb u =-\nu\nabla\theta & \quad \text{ in } \ot,\\
\theta_t - \Delta \theta =- \nu\theta \div\bb{u}_t& \quad \text{ in }  \ot,\\
\bb u\cdot\bb n=0,\ \N\bb u=0,\ \nabla\theta\cdot\bb n=0 &\quad\text{ on } (0,\infty)\times\partial\Omega,\\
\bb{u}(0,\cdot)=\bb{u}_0,\ \bb{u}_t(0,\cdot) = \bb{v}_0,\ \theta(0,\cdot)=\theta_0>0,&
\end{cases}
\end{equation}
where the operator $\N$ is defined as follows
\begin{align*}
\N\bb u=
\begin{cases}
\curl\bb u& \textrm{for }d=2,\\
\curl\bb u\times\bb n&\textrm{for }d=3.
\end{cases}
\end{align*}
Functions $\bb u\colon [0,\infty)\times\Omega\to\rn$ and $\theta\colon [0,\infty)\times\Omega\to\rn$ are unknown. We interpret $\bu$ as a displacement of the body and $\theta$ as its temperature. The existence of a global-in-time, strong, and unique solution for small initial data, as well as the long-time asymptotic behavior of solutions, are the main results of the paper.

Let us assume that $d=2$ for a while. For given vector field $\bb w=(w^1,w^2)\colon\Omega\to\r^2$, we define $\curl\bb w=w_{x_1}^2-w^1_{x_2}$. On the other hand, for a scalar function $g\colon\Omega\to\r$, we write $\curl g=\nabla^{\perp} g=(g_{x_2},-g_{x_1})$. Then, a vector field $\bb w\colon\Omega\to\r^2$ satisfies an analogous formula as vector fields for $d=3$, i.e.
\begin{align*}
\curl\curl\bb w=\nabla\div\bb w-\Delta\bb w.
\end{align*}

The constant $\nu\in\r$ is given. Constants $\lambda$ and $\mu$ are also fixed and are so-called the Lam\'e moduli (see \cite{Rackebook, Slaughter}). They satisfy
\begin{align*}
\mu>0\ant 2\mu+d\lambda>0.
\end{align*}
Other constants are assumed to be equal $1$ for simplicity. The operator
\begin{align}\label{oplame}
(\wspnd)\nabla\div\bu-\mu\curl\curl\bu
\end{align}
is called the Lam\'e operator. It can also be written as
\begin{align*}
\div\left(2\mu\bs\epsilon(\bu)+\lambda\div\bu I\right),
\end{align*}
where $\bs\epsilon(\bu):=\ul\left(\nabla\bu+\nabla^T\bu\right)$ is the linearized deformation tensor (see \cite{Temamcont, Slaughter}). In this form, we see a symmetry of the operator. However, for our purposes, the form in \eqref{oplame} is more suitable. We can treat the Lam\'e operator as a generalization of the Laplacian. It is more appropriate to consider the former than the latter from an engineering point of view.

We posit the Neumann boundary condition for the temperature. It is a natural setting in this context (see \cite{CMT, Rackebook}). It means that the heat cannot escape from the body. The standard boundary condition for the displacement is of the Dirichlet type (see \cite{Rackebook, Winkler1, Winkler2, Hrusa2, TC2, CMT, Racke1, Racke2} and others). On the other hand, the Neumann boundary condition for $\bu$ is also studied (see \cite{Rackebook, Racke1, Slem}) or even the Neumann condition for the stress tensor $\bs\sigma:=2\mu\bs\epsilon(\bu)+\lambda\div\bu I-\nu\theta I$. We and T. Cie\'slak considered the Dirichlet condition for $\bu$ in one-dimensional domains in \cite{BC, BC2}. We prove the existence and uniqueness of global strong solutions and investigate the time-asymptotic behavior. Our methods are strongly based on estimates, including the Fisher information. The application of them to the system with the Dirichlet condition for the displacement in higher dimensions presents certain technical challenges related to boundary integrals. Hence, first, we, in cooperation with T. Cie\'slak, M. Fuest, J. Lankiet, B. Muha, and S. Trifunovi\'c, studied the problem on the multidimensional torus in \cite{Biest}. We omitted issues interlocked with the boundary conditions in this manner.

J. B\v{r}ezina and E. Feireisl study the compressible magnetohydrodynamics system in 3d in \cite{fei}. They propose a few physically grounded boundary conditions for the magnetic field vector. One of these is the conditions we have set for displacement, i.e.
\begin{align}\label{ubound}
\bu\cdot\bb n=0\ant\N\bu=0 \quad\textrm{ on }\partial\Omega.
\end{align}
Furthermore, the authors state that the Dirichlet boundary condition for the magnetic field renders the compressible magnetohydrodynamic system overdetermined. Perhaps, the same issue we have in thermoelasticity.
 
It is novel to posit \eqref{ubound} for the displacement in the context of thermoelasticity. Nevertheless, they fit well with our methods. They enable us to prove the same results as in \cite{Biest}. Indeed, we are able to close calculations with a functional
\begin{align*}
\F&(\bu,\bu_t,\theta):=\\
&\ul\left(\io (\div \bu_t)^2dx+\io |\curl \bu_t|^2dx+(\wspnd)\io|\nabla\div \bu|^2dx+\mu\io|\curl\curl\bu|^2dx+\io\frac{|\nabla\theta|^2}{\theta}dx\right)
\end{align*}  
involving the Fisher information and the higher-order derivatives of $\bu$  thanks to \eqref{ubound}. In addition, the displacement equipped with \eqref{ubound} responds well to the Helmholtz decomposition. It means that $\bu$ can be decomposed into the curl-free part and the divergence-free part, which satisfy the appropriate equations. Let us point out that the functional $\F$ is a key tool in our analysis.

System \eqref{system} emerges as the simplest nonlinear system in the context of thermoelasticity. It is coherent with the following Helmholtz free energy function
\begin{align*}
H(\nabla\bu,\theta)=\frac\lambda2(\div\bu)^2+\mu|\bs\epsilon(\bu)|^2+\theta-\theta\log\theta-\nu\theta\div\bu.
\end{align*}
We refer the reader to \cite{Rackebook, CMT, Slem} for a rigorous derivation of the model. Let us perceive that the system exhibits the conservation of energy. Indeed, we have the following proposition.
\begin{prop}\label{energy_cons}
Regular solutions $\theta, \bb u$ of \eqref{system} satisfy
\begin{align}\label{balance}
\begin{split}
&\ul\io \bb |u_t|^2dx+\frac\wspnd2\io (\div\bb u)^2dx+\frac\mu2\io|\curl \bb u|^2dx+\io\theta dx\\
=&\ul\io \bb |v_0|^2dx+\frac\wspnd2\io (\div\bb u_0)^2dx+\frac\mu2\io|\curl \bb u_0|^2dx+\io\theta_0 dx.
\end{split}
\end{align}
\end{prop}
\begin{proof}
Let us multiply the first equation in \eqref{system} by $\bb u_t$ and integrate over $\Omega$. Integrating by parts, thanks to the boundary conditions for $\bu$, gives us
\begin{align}\label{lameq1}
\ddt\left(\ul\io \bb{u}_t^2dx+\frac{2\mu+\lambda}2\io (\div \bb{u})^2dx+\frac\mu2\io|\curl \bb u|^2dx\right)&= -\nu\io\nabla\theta\cdot\bb u_tdx\nonumber\\
&=\nu\io\theta\div\bb u_tdx.
\end{align}
Next, we integrate the lower equation in \eqref{system} over $\Omega$. Due to the boundary conditions for $\theta$, we obtain
\begin{align*}
\ddt\io\theta dx=-\nu\io\theta\div\bb u_tdx.
\end{align*}
We add the above equality to \eqref{lameq1}. It asserts the claim.
\end{proof}

In the definition below, we explain what solutions to equation \eqref{system} we will be looking for. First, we introduce the following function space
\begin{align*}
\H:=\left\{\bb w\in\hom\colon \bb w\cdot\bb n=0\textrm{ on }\partial\Omega\right\}
\end{align*}
with the norm
\begin{align}\label{defnorm}
\|\bb w\|_{\H}:=\left(\|\div\bb w\|_{\ld}^2+\|\curl\bb w\|_{\ld}^2\right)^{\sul}.
\end{align}
We refer to Subsection \ref{sublap} for further information about $\H$ and its norm.
\begin{defi}\label{soldef}
We say that $(\bb u,\theta)$ is a solution to (\ref{system}) if:
\begin{itemize}
\item The initial data is of regularity
\begin{align*}
\bb u_0\in \htwo\cap\H,\quad \bb v_0\in\H,\quad \theta_0\in\hom\cap L^{\infty}(\tn).
\end{align*}
We also require that there exists $\tilde{\theta}>0$ such that $\theta_0(x)\geq\tilde{\theta}$ for all $x\in\tn$.
\item Solutions $\theta$ and $\bu$ satisfy
\begin{align}\label{reggl}
\begin{split}
\bu&\in L^{\infty}(0,\infty;\htwo\cap\H)\cap W^{1,\infty}(0,\infty;\H)\cap W^{2,\infty}(0,\infty;\ld),\\
\theta&\in L^{\infty}(0,\infty;\hom)\cap L^2_{\loc}(0,\infty;\htwo)\cap H^1_{\loc}(0,\infty; \ld).
\end{split}
\end{align}
\item The momentum equation
\begin{align}\label{defmom}
 \io\bb u_{tt}\cdot\bs\phi dx+(\wspnd)\io\div\bb u\div\bs\phi dx+\mu\io\curl\bb u\cdot\curl\bs\phi dx =-\nu\io\nabla\theta\cdot\bs\phi dx
\end{align}
is satisfied for almost all $t\in (0,\infty)$ and $\bs\phi\in\H$.
\item The entropy equation
\begin{align}\label{defheat}
\io\theta_t\psi dx+\io\nabla\theta\cdot\nabla\psi=-\nu\io\theta \div\bb u_{t}\psi dx
\end{align}
holds for almost all $t\in (0,\infty)$ and for all $\psi\in\hom$.
\item Initial conditions are attained in the following sense:
\[
\theta \in C([0,\infty);\hom),
\]
\[
u\in C([0,\infty);\H),\ u_t \in C([0,\infty);\ld).
\]
\end{itemize}
\end{defi}

\begin{rem}
Let $\bu$ and $\theta$ be as in Definition \ref{soldef}. Then, it can be easily shown that they satisfy the momentum equation and the heat equation in \eqref{system} pointwise almost everywhere on $\ot$. Moreover, the boundary conditions are satisfied in the sense of the trace.
\end{rem}

We establish the existence and uniqueness of solutions for small initial values of \eqref{system} in the first main result of the article. We formulate it in the following theorem.
\begin{tw}\label{existglm}
There exists a constant $D>0$, dependent on $d$, $|\nu|$, and $\tn$, such that for any initial values satisfying
\begin{align*}
\initcond
\end{align*} there exists an unique solution to problem \eqref{system} as described in Definition \ref{soldef}. Moreover, we have that there exist constants $C,c>0$ such that
\begin{align}\label{tezexglm}
c\leq\theta\leq C\quad\textrm{for almost all }(t,x)\in [0,\infty)\times\Omega.
\end{align}
\end{tw}

In the second main result, we characterize the long-time behavior of solutions obtained in Theorem \ref{existglm}. The Helmholtz decomposition plays a crucial role here. It turns out that the potential part of the displacement behaves completely different than the divergence-free part as time approaches infinity. The latter demonstrates the independence of the former, as well as of the temperature. It will be shown that it satisfies the homogeneous wave equation and exhibits oscillatory behavior. We have a different issue with the curl-free part of the displacement. It is strongly coupled to the temperature. The heat propagation process causes the damping effect, which eventually leads to the vanishing of the rotation-free part of $\bu$. On the other hand, the temperature tends to a constant function. An analogous asymptotic behavior we have in the case of \eqref{system} on the torus (see \cite{Biest}). Let us also notice that the potential part of $\bu$ and the temperature exhibit the same asymptotics as $u$ and $\theta$ in the 1d case (see \citep{BC2}). Thus, we see that they behave as the second principle of thermodynamics predicts. We denote the rotation-free part of $\bu$, $\bu_0$ and $\bb v_0$  as $\bs\chi$, $\bs\chi_0$ and $\tilde{\bs\chi}$, respectively.  Let us formulate our result regarding the long-time behavior of solutions of \eqref{system}.
\begin{tw}\label{twassm}
Let us take the initial data such that 
\begin{align}\label{assyassu}
\initcond,
\end{align}
where the constant $D$ is taken from Theorem \ref{existglm}. Moreover, let us assume that $\bu$ and $\theta$ are solutions of \eqref{system} in the sense of Definition \ref{soldef}, which start respectively from initial data $\bu_0$, $\bb v_0$ and $\theta_0$. Then, the following convergences hold
\begin{align*}
\bs\chi(t,\cdot)&\to 0\textrm{ in }\H \textrm{ when }t\to\infty,\\
\bs\chi_t(t,\cdot)&\to 0\textrm{ in }\ld \textrm{ when }t\to\infty,\\
\theta(t,\cdot)&\to \theta_{\infty}\textrm{ in }\ld \textrm{ when }t\to\infty,
\end{align*}
where $\theta_{\infty}:=\left(\ul\io \tilde{\bs\chi}_0^2dx+\ul\io |\nabla\bs\chi_0|^2dx+\io\theta_0 dx\right)/|\Omega|$.
\end{tw}

Let us discuss the literature. As we mentioned above, problem \eqref{system} is being studied for the first time with the proposed boundary conditions for $\bu$. However, the study of similar systems with different boundary conditions has a long history. First research was fairly devoted to studying linear thermoelasticity (see, for instance, \cite{lin1, lin2, Rackebook} and references therein). One of the first~approaches to the nonlinear models was paper \cite{Slem}. The author proved the existence of local-in-time solutions for the one-dimensional nonlinear thermoelsticity. Moreover, the existence of global solutions has also been established for initial data near steady states. Further results concerning the local-in-time existence or the existence of global ones in a neighborhood of steady states with the nonnegative temperature were obtained in  \cite{Racke1, Racke2, Racke3, Hrusa2}. 

Until recently, there were no results concerning the global existence of solutions for nonlinear problems, even in one-dimensional cases. However, the authors got the existence and strong-weak uniqueness of global measure-valued solutions in \cite{CMT}. We would like to emphasize that this result applies to higher dimensions. The global existence of unique solutions with a positive temperature in 1d was obtained in paper \cite{BC}. In article \cite{Biest}, thermoelasticity was studied on the multidimensional torus. The authors obtained the global unique solutions, as well as the positivity of temperature, for small initial data. Winkler approached the~existence of weak solutions of the thermoviscoelasticity model in 1D in \cite{Winkler2, Winkler1}. Moreover, thermoviscoelasticity in higher domains has been researched in \cite{Owczarek1, Owczarek}.  The finite-time blow-up was demonstrated for a nonlinear problem in the 1d case in \cite{DafHsi}. However, the authors studied the system with nonlinearity $(p(u_x)u_x)_x$ instead of the Laplacian (or the Lam\'e operator). The result was extended to a more general case (and a more physically appropriate one) in \cite{Hrusa1}. 

Regarding the long-time behavior, we should mention works \cite{Slem, Rackebook, TC2}. The asymptotics near steady states in a one-dimensional setting was studied there. The complete long-time behavior for the simplest nonlinear model for the one-dimensional setting was demonstrated in \cite{BC2}. Moreover, the authors proved that the convergence to equilibrium is in exponential rate in \cite{thermorate}. The full study of the time-asymptotics for higher dimensions is in \cite{Biest}. However, the authors deal with the set without the boundary, i.e., the torus.

At the end of the introduction, let us sketch the plan of the paper. We cite the general results which we shall use in the further part of the article in Section \ref{prel}. In Subsection \ref{subent}, we analyze basic facts about the entropy of \eqref{system}. This result was already done in \cite{CMT}, but not with boundary conditions \eqref{ubound} for $\bu$. Therefore, we decided to repeat the proof. Subsection \ref{sublap} is devoted to the study of the Poisson equation with boundary conditions \eqref{ubound}. It gives us the sequence of eigenfunctions of the Laplacian which satisfy \eqref{ubound}. This, in turn, brings us the basis of $\ld$ and the basis of $\H$, which satisfies \eqref{ubound}. The main burden of proof for the existence of solutions is found in Section \ref{druga}. We prove a priori estimates there. The functional $\F$ is our main tool. We prove estimates that are independent of time for initial data such that $\F\leq D$. The constant $D$ occurs during the reasoning. To close the estimates, we require the $L^{\infty}$-estimate for the temperature. We prove it via the Moser iteration. Consequently, we obtain a bound for the temperature from below. 

Then, we proceed to the proof of the global existence for small data. We establish it in Section \ref{secexst}. First, in Subsection \ref{subapp}, we construct the sequence of approximate solutions. We employ the half-Galerkin method there. The same technique was applied in \cite{CMT}. The earlier investigation of the eigenfunctions of the Laplacian with \eqref{ubound} enables the half-Galerkin construction. In Subsection \ref{subex}, we prove the existence of global solutions of \eqref{system} for the initial values such that $\F\leq D$. The constant $D$ is taken from Section \ref{druga}. The estimates in this section allow transition to the limit with the sequence of the approximate solutions. Then, it is shown that the limit functions are solutions. The final subsection, Subsection \ref{subuniq}, is devoted to proving the uniqueness of solutions.

In the last section, Section \ref{czwarta}, we study the long-time behavior of solutions of \eqref{system}. First, in Subsection \ref{subdec}, we apply the Helmholtz decomposition to $\bu$. It occurs that the displacement responds very well to Helmholtz decomposition. It is shown that the divergence-free part satisfies the homogeneous wave equation. On the other hand, we prove that the potential part of $\bu$ and the temperature are coupled in analogous equations as \eqref{system}. Finally, in Subsection \ref{subas}, first, we show that the divergence-free part oscillates when $t\to\infty$. Next, the decay of the curl-free part of the displacement and the convergence of the temperature to a constant function at infinity are proven. The utility of the second law of thermodynamics is very important in this last part.

\section{Preliminaries}\label{prel}
This section is devoted to introducing general results that will be utilized later. We start by citing the lemma from \cite{EvansEntropy}. It will be crucial in proving Lemma \ref{wonderfulformula}. It states that a certain integral, which will appear in the proof, is nonnegative. 
\begin{lem}\label{Evansentr}
Let $\Omega\subset\rn$ be a smooth, bounded, and convex set. Suppose that $v\in C^2(\overline\Omega)$ and 
\begin{align*}
\frac{\partial v}{\partial\bb n}=0\textrm{ on }\partial\Omega.
\end{align*}
Then, the inequality
\begin{align*}
\frac{\partial|\nabla v|^2}{\partial\bb n}\leq 0\textrm{ on }\partial\Omega
\end{align*}
holds.
\end{lem}
The next lemma presents the unusual inequality between the norms of the gradient of a function and its Laplacian. The proof can be found in \cite{Fuestin}.
\begin{lem}\label{inpion}
If $v\in \htwo$ and $\frac{\partial v}{\partial \bb n}=0$ on $\partial\Omega$, then the following inequality
\begin{align*}
\|\nabla v\|_{\ld}\leq C\|\Delta v\|_{\ld}
\end{align*}
is satisfied and $C$ depends only on $d$ and $\Omega$.
\end{lem}

As in previous papers about thermoelasticity (\cite{BC2, Biest, BC}), the following inequality will also be important. It estimates the integral that will appear in the calculations. The proof can be found in \cite{CFHS}.
\begin{lem}\label{theCin}
Let $w\in C^2(\Omega)$ be a positive function such that $\frac{\partial w}{\partial\bb n}=0$ on $\partial\Omega$, then we have
\begin{align*}
\io |\nabla^2w^{\sul}|^2dx\leq C\io w|\nabla^2\log w|^2dx,
\end{align*}
where $C=1+\frac{\sqrt d}2+\frac d8$.
\end{lem}
In the following theorem, the Helmholtz decomposition of a vector field is stated. It will be applied to decompose the displacement in Subsection \ref{subdec}. For the proof, see e.g. \cite{FN}.
\begin{tw}\label{thdecomp}
Let us assume that $\bb v\in H^k(\tn)$ is a vector field, where $k\in\n\cup\{0\}$. Then, $\bb v$ can be decomposed uniquely as follows
\begin{align*}
\bb v=H\bb v+\hp\bb v,
\end{align*}
where $H\bb v, \hp\bb v\in H^k(\tn)$ satisfy
\begin{equation*}
\div H\bb v=0\ant H\bb v\cdot\bb n\textrm{ on }\tn,
\end{equation*}
and there exists function $\phi\in H^{k+1}(\tn)$ such that
\begin{align*}
\io \phi\, dx=0\quad\textrm{ and }\quad\hp \bb v=\nabla\phi.
\end{align*}
Moreover, $H\bb v$ and $\hp\bb v$ are orthogonal in $L^2(\tn)$.
\end{tw}

\subsection{Entropy}\label{subent}

We introduce simple facts about the entropy of \eqref{system} in this subsection. Let us assume that $\bu$ and $\theta$ are solutions of \eqref{system} in the sense of Definition \ref{soldef}. Moreover, let us assume for a while that $\theta>0$ (it will be proven below). Then, we can study the evolution of the entropy. Indeed, let us denote $\tau:=\log\theta$. Then, the following equation is satisfied
\begin{align}\label{eqtau}
\tau_t-\Delta\tau-|\nabla\tau|^2=-\nu\div u_t
\end{align}
for almost all $(t,x)\in (0,T)\times\tn$.
An analogous line of reasoning to that in \cite{CMT} leads us to the following proposition.
\begin{prop}\label{ogrzprop}
Let $\theta>0$ and $\bu$ be solutions of the problem \eqref{system} in the sense of Definition \ref{soldef}. Then, the following equation
\begin{align}\label{sepri}
\ddt\io\tau dx=\io|\nabla\tau|^2dx
\end{align}
is satisfied for almost all $t\in (0,\infty)$. Moreover, we have that
\begin{align*}
\nabla\tau\in L^2((0,\infty)\times \tn)
\end{align*}
and
\begin{align*}
\tau\in L^{\infty}(0,\infty;L^1(\tn)).
\end{align*}
\end{prop}
\begin{proof}
Equation \eqref{sepri} follows immediately from Equation \eqref{eqtau}. Indeed, we integrate \eqref{eqtau} over $\Omega$ and utilize the boundary conditions for $\bu$ and $\theta$. Next, let us fix $t>0$ and let us integrate \eqref{sepri} over $[0,t]$. It gives that 
\begin{align*}
-\io\tau dx+\int_0^t\io|\nabla\tau|^2dxds=-\io\tau_0dx,
\end{align*}
where $\tau_0:=\ln\theta_0$. We add the above inequality to \eqref{balance}. It gives that
\begin{align*}
\ul\io \bb u_t^2dx+\frac\wspnd2\io (\div\bb u)^2dx+\frac\mu2\io|\curl \bb u|^2dx+\io(\theta-\tau) dx+\int_0^t\io|\nabla\tau|^2dxds\\
=\ul\io \bb v_0^2dx+\frac\wspnd2\io (\div\bb u_0)^2dx+\frac\mu2\io|\curl \bb u_0|^2dx+\io(\theta_0-\tau_0)dx.
\end{align*}
The last integral on the right-hand side is finite and positive, so integrals on the left-hand side are bounded for almost all $t$. It asserts the thesis.
\end{proof}

Let us perceive that \eqref{sepri} is a quantitative version of the second principle of thermodynamics. In Section \ref{czwarta}, it will play a crucial role in demonstrating that the $\omega$-set of system \eqref{system} consists of a single element.

\subsection{Eigenfunctions of vector Laplacian}\label{sublap}
In this subsection, we will study the eigensystem of the following problem: we want to find $\bb w\colon\Omega\to\rn$ and $\xi\in\r$
\begin{align}\label{probeig}
\begin{cases}
-\Delta\bb w=\xi\bb w, & \textrm{in } \Omega,\\
\bb w\cdot\bb n=0, & \textrm{on } \partial\Omega,\\
\N\bb w=0, & \textrm{on } \partial\Omega. 
\end{cases}
\end{align}
Thus, we want to find eigenfunctions of $-\Delta$ with the boundary conditions for the displacement in \eqref{system}. They will allow us to construct the sequence of the approximate solutions of \eqref{system} in Subsection \ref{subapp}.

Let us remind the space
\begin{align*}
\H:=\left\{\bb v\in\hom\colon \bb v\cdot \bb n=0 \textrm{ on }\partial\Omega\right\}.
\end{align*}
We define a bilinear form $B\colon\H\times\H\to\r$,
\begin{align}\label{scprod}
B(\bb v,\bb w):=\io\div\bb v\div\bb wdx+\io\curl\bb v\cdot\curl\bb wdx.
\end{align}
It is obvious that $B$ is bounded. Moreover, Theorem IV.4.8 in \cite{Boyer} and Remark 3.5 of Chapter I in \cite{G-R} guarantee that $B$ is coercive on $\H$ with the norm $\|\cdot\|_{\hom}$. Hence, by the Lax-Milgram Theorem, we obtain that for any $\bb f\in\ld$ there exists a unique $\bb w\in\H$ such that 
\begin{align*}
B(\bb w, \bb v)=(\bb f,\bb v)_{\ld}\quad \textrm{ for all }\bb v\in\H.
\end{align*}

It is easy to see that the above formula is a weak formulation of the following problem
\begin{align}\label{veclap}
\begin{cases}
-\Delta\bb w=\bb f, & \textrm{in } \Omega,\\
\bb w\cdot\bb n=0, & \textrm{on } \partial\Omega,\\
\N\bb w=0, & \textrm{on } \partial\Omega. 
\end{cases}
\end{align}
Moreover, Theorem IV.9.6 of \cite{Boyer} gives us that $\bb w\in\htwo$. Let us notice that there is a proof only for $d=3$, but it is analogous in the case when $d=2$. Therefore, we can state the following result.
\begin{tw}\label{prexpr}
Let us assume that $\tn$ is a simply connected domain of $\rn$ with a boundary of class $C^{2,1}(\tn)$. Then, for every $\bb f\in\ld$ there exists a unique $\bb w\in\htwo\cap\H$ such that $\bb w$ is a solution of $\eqref{veclap}$. The Poisson equation in \eqref{veclap} is satisfied pointwise for almost all $x\in\Omega$, and the boundary conditions are satisfied in the sense of traces. Moreover, the following inequality is satisfied
\begin{align*}
\|\bb w\|_{\htwo}\leq C\|\bb f\|_{\ld},
\end{align*}
where $C>0$ depends on $\Omega$ and $d$.
\end{tw}
Let us note that the coercivity of the bilinear form in \eqref{scprod} entails that the norm $\|\cdot\|_{\H}$ from \eqref{defnorm} is equivalent to the standard norm on $\hom$.
 
Now, we can apply the standard theory of symmetric and compact operators in Hilbert spaces (see \cite[Theorem 7 of Appendix D]{Evans}) to the operator solving problem \eqref{veclap}. As an inference, we get the existence of an orthogonal basis (with $B$ defined in \eqref{scprod} as scalar product) $\{\bs\phi_k\}$ of $\H$ consisting of eigenfunctions as in \eqref{probeig}, which in addition is an orthonormal basis in $\ld$. Let us denote the eigenvalues of \eqref{probeig} by $\xi_k$. Because $\bs\phi_k$ satisfy $-\Delta\bs\phi_k=\xi_k\bs\phi_k$ in the interior of $\Omega$, due to the standard results for the regularity of solutions of elliptic equations (see \cite{Boyer, Evans}), we get that $\bs\phi_k\in C^{\infty}(\Omega)$. However, Theorem \ref{prexpr} says that $\bs\phi_k\in\htwo$ for all $k\in\n$. Albeit applying similar methods to those in \cite{Boyer}, it can be shown that the eigenfunctions are arbitrarily well-regular as long as the boundary's regularity allows it.

\section{A priori estimates}\label{druga}
In this section, we derive estimates that will be used to prove the existence of global solutions in the next section. Let us remind the functional 
\begin{align*}
\F&(\bu,\bu_t,\theta):=\\
&\ul\left(\io (\div \bu_t)^2dx+\io |\curl \bu_t|^2dx+(\wspnd)\io|\nabla\div \bu|^2dx+\mu\io|\curl\curl\bu|^2dx+\io\frac{|\nabla\theta|^2}{\theta}dx\right).
\end{align*}
It plays a crucial role in estimates. It consists of the Fisher information and higher-order derivatives of the displacement. Actually, this is a version of the functional that has been adopted to address our issue. Similar ones can be found in \cite{BC, Biest, BC2}. We start proving the inequality in the following lemma.
\begin{lem}\label{wonderfulformula}
Let $\bu$ and $\theta$ be smooth solutions of \eqref{system} such that $\theta>0$. Then, the following inequality
\begin{align*}
\ddt\F(\bu,\bu_t,\theta)
\leq-\io\theta|\nabla^2\log\theta|^2dx-\frac\nu2\io\frac{|\nabla\theta|^2}\theta\div \bu_tdx
\end{align*}
holds.
\end{lem}
\begin{proof}
Direct calculations based on \eqref{system} yield
\begin{align*}
(\theta^{\sul})_t=\Delta\theta^{\sul}+\frac14\frac{|\nabla\theta|^2}{\theta^{\sfrac32}}-\frac{\nu}2\theta^{\sul}\div \bu_t.
\end{align*}
We multiply the above equation by $-4\Delta\theta^{\sul}$. After integration, we obtain that the left-hand side is equal to
\begin{align*}
-4\io(\theta^{\sul})_t\Delta\theta^{\sul}dx=4\io\nabla\theta^{\sul}\cdot\nabla(\theta^{\sul})_tdx=2\ddt\io |\nabla\theta^{\sul}|^2dx=\ul\ddt\io\frac{|\nabla\theta|^2}{\theta}dx.
\end{align*}
Let us deal with the right-hand side. We obtain
\begin{align}\label{surow1}
\ul\ddt\io\frac{|\nabla\theta|^2}{\theta}dx&=-4\io(\Delta \theta^{\sul})^2dx-\io\Delta\theta^{\sul}\frac{|\nabla\theta|^2}{\theta^{\frac32}}dx+2\mu\io\Delta\theta^{\sul}\theta^{\sul}\div \bb{u}_tdx \nonumber\\
&= I_1+I_2+I_3.
\end{align}

First, we calculate $I_1$. The integration by parts yield
\begin{align*}
I_1=4\io\Delta \nabla\theta^{\sul}\cdot\nabla\theta^{\sul}dx.
\end{align*}
We apply the Bochner formula $\Delta \nabla\theta^{\sul}\cdot\nabla\theta^{\sul}=\ul\Delta |\nabla\theta^{\sul}|^2-|\nabla^2\theta^{\sul}|^2$ and arrive at
\begin{align*}
I_1=-4\io|\nabla^2\theta^{\sul}|^2dx+2\io\Delta |\nabla\theta^{\sul}|^2dx.
\end{align*}
Lemma \ref{Evansentr} entails
\begin{align}\label{surow2}
I_1\leq -4\io|\nabla^2\theta^{\sul}|^2dx
\end{align}

We integrate by parts in $I_2$ and obtain
\begin{align*}
I_2=4\io\nabla\theta^{\sul}\cdot\nabla\frac{|\nabla\theta^{\sul}|^2}{\theta^{\sul}}dx=8\io\frac{\nabla\theta^{\sul} \cdot \nabla^2\theta^{\sul}\cdot(\nabla\theta^{\sul})^T}{\theta^{\sul}}dx-4\io\frac{|\nabla\theta^{\sul}|^4}{\theta}dx
\end{align*}
In $I_3$ we get
\begin{align}\label{surow4}
I_3=\nu\io\Delta\theta\div \bu_tdx-\frac{\nu}2\io\frac{|\nabla\theta|^2}{\theta}\div \bu_t dx.
\end{align}
Let us note that
\begin{align}\label{surow5}
-4\io|\nabla^2\theta^{\sul}|^2dx+I_2&=-4\io\sum_{k,l=1}^n\left((\theta^{\sul})_{x_kx_l}-\frac{(\theta^{\sul})_{x_k}(\theta^{\sul})_{x_l}}{\theta^{\sul}}\right)^2dx\nonumber\\
&=-4\io\theta|\nabla^2\log\theta^{\sul}|^2dx
\end{align}
Plugging \eqref{surow2}, \eqref{surow4} and \eqref{surow5} into \eqref{surow1}, we get
\begin{align}\label{wfe}
\ul\ddt\io\frac{|\nabla\theta|^2}{\theta}dx\leq-\io\theta|D^2\log\theta|^2dx-\frac\nu2\io\frac{|\nabla\theta|^2}\theta\div \bu_tdx+\nu\io\Delta\theta\div \bu_tdx.
\end{align}

Let us remind that for a vector field $\bb w\colon\Omega\to\rn$ we have
\begin{align}\label{identity}
\Delta \bb w=-\curl\curl \bb w+\nabla\div \bb w.
\end{align}
We multiply the wave equation in \eqref{system} by $-\Delta \bu_t$ and integrate over $\Omega$. Utilizing the above formula and integrating by parts, we obtain
\begin{align*}
\ul\ddt&\left(\io(\div \bu_t)^2dx+\io|\curl \bu_t|^2dx+(\wspnd)\io|\nabla\div\bu|^2dx+\lambda\io|\curl\curl\bu|^2dx\right)\\
&=\nu\io\nabla\theta\cdot\Delta \bu_t dx=\nu\io\nabla\theta\cdot\nabla\div \bu_t dx-\nu\io\nabla\theta\cdot\curl\curl\bu_t=-\nu\io\Delta\theta\div\bu_tdx.
\end{align*}
The integral $-\nu\io\nabla\theta\cdot\curl\curl\bu_tdx$ disappears after integration by parts and the utility of boundary conditions of $\bb u_t$. We add the above equality to \eqref{wfe}. The claim follows.
\end{proof}
In the lemma below, we continue to estimate $\F$.
\begin{lem}
Let $\bu$ and $\theta$ be smooth solutions of \eqref{system} such that $\theta>0$. Then, the following inequality is satisfied
\begin{align}\label{Ffestin}
\ddt\F(\bb u,\bu_t, \theta) &\leq -C_1 \|\nabla^2\theta^{\sul}\|_{\ld}^2 + C_2 \F^{\sul}(\bb u,\bu_t, \theta) \|\nabla \theta^{\sul}\|_{L^4(\Omega)}^2,
\end{align}
where $C_1, C_2 > 0$ only depend on $d$, $|\nu|$ and $\Omega$.
\end{lem}
\begin{proof}
By Lemma~\ref{wonderfulformula}, we get
\begin{align*}
\ddt \F(\bb u,\bu_t, \theta) \leq-\io\theta|\nabla^2\log\theta^{\sul}|^2dx-\frac\mu2\io\frac{|\nabla\theta|^2}\theta\div \bb{u}_tdx =I_1 + I_2.
\end{align*}
The term $I_1$ can be estimated by Lemma~\ref{theCin}. It gives
\begin{align*}
I_1\leq -C_1\|\nabla^2\theta^{\sul}\|_{\ld}^2.
\end{align*}
On the other hand, the Hölder inequality yields
\begin{align*}
I_2&\leq   C_2\|\nabla\theta^{\sul}\|_{L^4(\Omega)}^2\|\div \bb{u}_t\|_{\ld}\leq    C_2 \F^{\sul}(\bb u,\bu_t, \theta) \|\nabla\theta^{\sul}\|_{L^4(\Omega)}^2.		
\end{align*}
\end{proof}

In the next lemma, we present one of the main components in the proof of the existence of solutions. The constant $D$ from the result below will be fundamental in the rest of the paper. 
\begin{lem}
Let us assume that $\bu$ and $\theta$ are smooth solutions of \eqref{system} such that $\theta>0$. Then,
\begin{align*}
\ddt\F\leq C\|\nabla^2\theta^{\sul}\|_{\ld}^2(\F-D),
\end{align*}
where $C$ and $D$ depend on $d$, $|\nu|$ and $\Omega$.
\end{lem}
\begin{proof}
We shall estimate the norm $\|\nabla\theta^{\sul}\|_{L^4(\Omega)}^2$ in \eqref{Ffestin}. We use the Gagliardo-Nirenberg inequality and the Young inequality to obtain
\begin{align*}
\|\nabla\theta^{\sul}\|^2_{L^4(\Omega)}\leq C_2(\|\nabla^2\theta^{\sul}\|_{\ld}^2+\|\nabla\theta^{\sul}\|_{\ld}^2).
\end{align*}
Now, we utilize Lemma \ref{inpion}. It gives us
\begin{align*}
\|\nabla\theta^{\sul}\|^2_{L^4(\Omega)}\leq C_2(\|\nabla^2\theta^{\sul}\|_{\ld}^2+\|\Delta\theta^{\sul}\|_{\ld}^2)\leq C_2 \|\nabla^2\theta^{\sul}\|_{\ld}^2
\end{align*}
We plug the above inequality into \eqref{Ffestin}, and obtain
\begin{align*}
\ddt\F\leq -C_1\|\nabla^2\theta^{\sul}\|_{\ld}^2+C_2\|\nabla^2\theta^{\sul}\|_{\ld}^2\F^{\sul}.
\end{align*}
We utilize the Young inequality and arrive at
\begin{align*}
\ddt\F&\leq (\epsilon-C_1)\|\nabla^2\theta^{\sul}\|_{\ld}^2+\frac{C_2}{\epsilon}\|\nabla^2\theta^{\sul}\|_{\ld}^2\F\\
&\leq (\epsilon-C_1)\|\nabla^2\theta^{\sul}\|_{\ld}^2+\frac{C_2}{\epsilon}\|\nabla^2\theta^{\sul}\|_{\ld}^2\F.
\end{align*}
Let us take $\epsilon:=\frac{C_1}2$. We obtain the assertion with $C=\frac{2C_2}{C_1}$ and $D=\frac{C_1^2}{4C_2}$.
\end{proof}

As an immediate consequence of the previous lemma, we obtain the following theorem. It provides the global estimates of $\F$ for solutions with initial values smaller than $D$.
\begin{tw}\label{estsol}
Let us assume that $\bu$ and $\theta$ are smooth solutions of \eqref{system} which satisfy $\theta>0$. There exists a constant $D=D(d,|\nu|,\Omega)$ such that if
\begin{align*}
\initcond,
\end{align*}
then
\begin{align*}
\F(\bu,\bu_t,\theta)\leq  \F(\bu_0,\bb v_0,\theta_0)
\end{align*}
for all $t\geq 0$.
\end{tw}
Let us notice that the right-hand side of the inequality in the above theorem is finite because the bound from below of $\theta_0$ is positive. We will need a global estimate for the temperature in $L^{\infty}$ to complete our estimates. This will be inferred as a consequence of the following general lemma. We will prove it utilizing the Moser-Alikakos iteration (see \cite{Moser, Ali}). We already applied this method in the context of thermoelasticity in \cite{BC2, Biest}. However, it had been earlier used for 1d thermoelasticity in \cite{DafHsi}.
\begin{lem}\label{lemmos}
Let us assume that $\eta$ is a smooth function defined on $\ot$ and $f\in L^{\infty}(0,\infty;\ld)$. Moreover, let us assume that $\eta$ satisfies
\begin{equation}\label{lemmosas}
\eta_t-\Delta\eta\leq (\eta^++1)f\qquad\textrm{for almost all $(t,x)\in (0,\infty)\times\Omega$}\footnote{The function $\eta^+$ means the nonnegative part of $\eta$ i.e. $\eta^+:=\max\{\eta,0\}$.}
\end{equation}
and 
\begin{align*}
\frac{\partial\eta}{\partial\bb n}=0\textrm{ on }\partial\Omega.
\end{align*} 
Then, there exists a constant $C>0$, dependent on $d$, $|\nu|$, $\Omega$, $\|f\|_{L^{\infty}(0,T; L^1(\Omega))}$, $\|\eta^+(0)\|_{L^{\infty}(\Omega)}$, and $\|\eta^+\|_{L^{\infty}(0,\infty;L^1(\tn))}$ such that
\begin{align*}
\|\eta^+\|_{L^{\infty}((0,\infty)\times\Omega)}\leq C.
\end{align*}
\end{lem}
\begin{proof}
Let us take $n\in\n$ and let us multiply \eqref{lemmosas} by $\left(\eta^+\right)^{2^{n+1}-1}$ and integrate over $\Omega$. After integration by parts, we obtain
\begin{align}\label{in1supth}
\frac1{2^{n+1}}\ddt\io(\eta^+)^{2^{n+1}}dx+\left(\frac1{2^{n-1}}-\frac1{4^n}\right)\io\left|\nabla\left((\eta^+)^{2^n}\right)\right|^2dx&\leq +\io \left((\eta^+)^{2^{n+1}}+(\eta^+)^{2^{n+1}-1}\right)f dx\nonumber\\
&\leq C( \|(\eta^+)^{2^n}\|_{L^4(\tn)}^2+1),
\end{align}
where in the last inequality we used the Young inequality. Now, we apply the Gagliardo-Nirenberg inequality
\begin{align*}
\|(\eta^+)^{2^n}\|_{L^4(\tn)}^2\leq C(\|\nabla\left((\eta^+)^{2^n}\right)\|_{\ld}^{2p}\|(\eta^+)^{2^n}\|_{L^1(\Omega)}^{2(1-p)}+\|(\eta^+)^{2^n}\|^2_{L^1(\Omega)}),
\end{align*}
where $p=\frac34$ for $d=2$ and $p=\frac9{10}$ for $d=3$. We apply the Young inequality to the above one and arrive at
\begin{align*}
\|(\eta^+)^{2^n}\|_{L^4(\Omega)}^2\leq C(\epsilon\|\nabla\left((\eta^+)^{2^n}\right)\|_{\ld}^{2}+\epsilon^{-\alpha}\|(\eta^+)^{2^n}\|_{L^1(\Omega)}^2),
\end{align*}
where $\alpha=3$ for $d=2$ and $\alpha=9$ for $d=3$. Let us denote $a_n:=\frac1{2^{n}}-\frac1{2\cdot4^n}$ and choose $\epsilon:=\frac{a_n}{2C}$. Plugging the above inequality into \eqref{in1supth}, we get
\begin{align*}
\frac1{2^{n+1}}\ddt\io(\eta^+)^{2^{n+1}}dx+a_n\io\left|\nabla\left((\eta^+)^{2^n}\right)\right|^2dx\leq C( a_n^{-\alpha}\|(\eta^+)^{2^n}\|_{L^1(\Omega)}^2+1).
\end{align*}
It leads us to
\begin{align*}
\ddt\io(\eta^+)^{2^{n+1}}dx+\io\left|\nabla\left((\eta^+)^{2^n}\right)\right|^2dx\leq C 2^{n+\alpha n}(\|(\eta^+)^{2^n}\|_{L^1(\Omega)}^2+1).
\end{align*}
Applying the Gagliardo-Nirenberg and the Young inequalities once more yields
\begin{align*}
\ddt \|\eta^+\|_{L^{2^{n+1}}(\Omega)}^{2^{n+1}}+C_1\|\eta^+\|_{L^{2^{n+1}}(\Omega)}^{2^{n+1}}\leq C_2 2^{n+\alpha n}(\|\eta^+\|_{L^{2^n}(\Omega)}^{2^{n+1}}+1).
\end{align*}

Solving the above differential inequality, we get
\begin{align*}
\|\eta^+\|_{L^{\infty}(0,\infty;L^{2^{n+1}}(\Omega))}^{2^{n+1}}&\leq 2^{n+\alpha n}C(\|\eta^+\|_{L^{\infty}(0,\infty;L^{2^{n}}(\Omega))}^{2^{n+1}}+1)+\|\eta^+(0)\|_{L^{2^{n+1}}(\Omega)}^{2^{n+1}}\\
&\leq C (2^{n+\alpha n}(\|\eta^+\|_{L^{\infty}(0,\infty;L^{2^{n}}(\Omega))}^{2^{n+1}}+1)+\|\eta^+(0)\|_{L^{\infty}(\Omega)}^{2^{n+1}}).
\end{align*}
Let us denote $m_n:=\max\{\|\eta^+\|_{L^{\infty}(0,\infty;L^{2^{n}}(\Omega))},\|\eta^+(0)\|_{L^{\infty}(\Omega)},1\}$. Then, the above inequality can be rewritten as follows
\begin{align*}
m_{n+1}^{2^{n+1}}\leq C(2^{n+\alpha n}+1)m_n^{2^{n+1}}.
\end{align*}
It gives us
\begin{align*}
m_{n+1}\leq 2^{\frac{n+\alpha n}{2^{n+1}}}C^{\frac1{2^{n+1}}}m_n
\end{align*}
for all $n\in\n$. Solving this recursion inequality yields
\begin{align*}
m_n\leq \prod_{k=1}^n C^{\frac1 {2^k}}2^{\frac{k\alpha}{2^k}}m_0\leq m_0C^{\sum_{k=1}^{\infty}\frac1{2^k}}2^{\sum_{k=1}^{\infty}2^{\frac{k\alpha}{2^k}}}.
\end{align*}
Because both series on the right-hand side of the above inequality are convergent, we get
\begin{align*}
\|\eta^+\|_{L^{\infty}(0,\infty;L^{2^{n+1}}(\Omega))}\leq C.
\end{align*}
Thus, we obtain
\begin{align*}
\|\eta^+\|_{L^{\infty}(0,\infty;L^{\infty}(\Omega))}\leq C.
\end{align*}
\end{proof}

Let us remind that $\|\bb w\|_{\H}=\left(\|\div\bb w\|_{\ld}^2+\|\curl\bb w\|_{\ld}^2\right)^{\sul}$.
\begin{tw}\label{upbofortemp}
Let $\bu$ and $\theta$ be smooth solutions of \eqref{system} such that $\theta>0$. Moreover, let us assume that the initial values satisfy the inequality
\begin{align*}
\initcond,
\end{align*}
where $D$ is taken from Theorem \ref{estsol}. Then,
\begin{align*}
\|\theta\|_{L^{\infty}((0,\infty)\times\tn)}\leq C,
\end{align*}
where $C=C(\Omega,d,|\nu|,\lambda,\mu,\| \bb v_0\|_{\H}, \|\bu_0\|_{\htwo}, \|\theta_0\|_{\hom}, \tilde\theta,\|\theta_0\|_{L^{\infty}(\Omega)})$.
\end{tw}
\begin{proof}
We utilize Lemma \ref{lemmos} for $\theta$ as $\eta$ and $|\nu||\div\bu_t|$ as $f$. Due the lower equation in \eqref{system}, we obtain
\begin{align*}
\theta_t-\Delta\theta\leq (\theta+1)|\nu||\div\bu_t|.
\end{align*}
By Theorem \ref{estsol}, we have an estimate for $\|\nu\div\bu_t\|_{\ld}$. Hence, we obtain the thesis, because $\theta^+=\theta$.
\end{proof}

In addition, the lower boundedness of the temperature can be derived from Lemma \ref{lemmos}. We do not need it in the proof of the existence of solutions. However, it will be important in our considerations about the time-asymptotics of \eqref{system}. Moreover, this is a valuable result in its own right.
\begin{tw}\label{downbotemp}
Let us assume $\bu$ and $\theta$ are smooth solutions of \eqref{system} with the positive temperature. Then, there exists the positive constant $C$ depedent on $\Omega$, $d$ ,$|\nu|$, $\lambda$, $\mu$, $\| \bb v_0\|_{\H}$, $\|\bu_0\|_{\htwo}$, $\|\theta_0\|_{\hom}$, $\tilde\theta$, and $\|\theta_0\|_{L^{\infty}(\Omega)}$ such that
\begin{align*}
\theta\geq C\qquad \textrm{ for almost all } (t,x)\in (0,\infty)\times\Omega.
\end{align*}
\end{tw}
\begin{proof}
Let us remind the notation $\tau:=\log\theta$. We want to estimate $\|\tau\|_{L^{\infty}((0,\infty)\times\Omega)}$. The bound for $\|\tau^+\|_{L^{\infty}((0,\infty)\times\Omega)}$ we have from Theorem \ref{upbofortemp}. Hence it is left to constrain the norm $\|\tau^-\|_{L^{\infty}((0,\infty)\times\Omega)}$, where $\tau^-:=\max\{0,-\tau\}$. We know that $\tau$ satisfies equation \eqref{eqtau}, so we obtain
\begin{align*}
(-\tau)_t-\Delta(-\tau)+|\nabla\tau|^2=\nu\div\bu_t.
\end{align*}
Hence, we have \eqref{lemmosas} for $\eta:=-\tau$ and $f:=|\nu||\div\bu_t|$. Consequently, by Lemma \ref{lemmos}, we assert that $\tau^-\in L^{\infty}((0,\infty)\times\tn))$. This infers that
\begin{align*}
\theta\geq e^{-\|\tau\|_{L^{\infty}((0\infty)\times\tn)}}.
\end{align*} 
\end{proof}

\section{Existence and uniqueness of solutions}\label{secexst}
In this section, we prove the existence and uniqueness of solutions. First, we construct the sequence of approximate solutions. We utilize the half-Galerkin method there. Then, we apply the results from Section \ref{druga} to the sequence to obtain the estimates. This will enable us to reach the limit in the sequence. In this way, we will get the existence. Then, we will prove the uniqueness of solutions. A similar method to build the solution was applied in \cite{Biest, BC, CMT}.

\subsection{Approximate problem}\label{subapp}

Let $\{\bs\phi_k\}$ denote the same sequence as in Subsection \ref{sublap}, i.e. let $\{\bs\phi_k\}$ be the sequence of eigienfunctions of \eqref{probeig}. We signify $V_n=\operatorname{span}\{\bs\phi_k\}_{1\leq k\leq n}$ for $n\in\n$.
\begin{defi}\label{apdef}
We say that $\bb u_n\in C^2([0,\infty);V_n)$ and $\theta_n\in L^2_{\loc}(0,\infty;\htwo)\cap H^1_{\loc}(0,\infty;\ld)$ are solutions of an approximate problem if for all $\bs\phi\in V_n$ and $\psi\in \hom$ the following equations are satisfied for almost all $t\in (0,\infty)$
\begin{align*}
\io \bb{u}_{n,tt}\cdot\bs\phi dx+(\wspnd)\io \div\bb{u}_{n}\div\bs\phi dx+\mu\io\curl\bb{u}_n\cdot\curl\bs\phi&=\nu\io\theta_n\div\bs\phi dx,\\
\io\theta_{n,t}\psi dx+\io \nabla\theta_{n}\cdot\nabla\psi dx&=-\nu\io\psi\theta_n \div \bb{u}_{n,t}dx.
\end{align*}
Moreover, the following equalities
\begin{align*}
\theta_n(0)=\theta_{0},\quad \bb u_n(0)=P_{V_n}\bb u_0,\quad \bb u_{n,t}(0)=P_{V_n}\bb v_0
\end{align*}
are satisfied. Where $P_{V_n}$ is an orthogonal projection of the corresponding spaces onto $V_n$.
\end{defi}
We proceed in an analogous way as in  \cite[Proposition 1 and Lemma 4.2]{CMT} to obtain the following result.
\begin{prop}
For every $n\in\mathbb{N}$, there exists a solution $(\bb u_n,\theta_n)$ of the approximate problem in the sense of Definition \ref{apdef}. Moreover, there exist $\hat\theta_n(t)>0$ such that $\theta_n(x,t)\geq\hat\theta_n(t)$ for almost all $(t,x)\in \ot$.
\end{prop}

The choice of $\{\bs\phi_k\}$ for a basis ensures that $\bu_n$ satisfies \eqref{ubound}.

\subsection{Existence of solutions}\label{subex}
In this section, we establish the existence of global solutions, provided that $\F(\bu_0,\bb v_0,\theta)\leq D$. We begin applying the estimates from Section \ref{druga} to the sequence of the approximate solutions.
\begin{tw}\label{twglest}
Let us assume that $(\bu_n,\theta_n)$ is a solution to an approximate problem in the sense of Definition \ref{apdef}. Moreover, let us assume that the initial values satisfy the inequality
\begin{align*}
\initcond
\end{align*}
where $D$ is taken from Theorem \ref{estsol}. Then, for any $T>0$ the following inequalities hold
\begin{align}
\|\bb u_n\|_{L^{\infty}(0,T;\htwo)}&\leq C, \label{t1glest}\\
\|\bb u_{n,t}\|_{L^{\infty}(0,T;\H)}&\leq C,\label{t2glest}\\
\|\bb u_{n,tt}\|_{L^{\infty}(0,T;\ld)}&\leq C,\label{t3glest}\\
\|\theta_n\|_{L^{\infty}(0,T;\hom)}&\leq C,\label{t4glest}\\
\|\theta_{n,t}\|_{L^{2}(0,T;\ld)}&\leq \widetilde C(T),\label{t5glest}\\
\|\theta_n\|_{L^2(0,T;\htwo)}&\leq \widetilde C(T),\label{t6glest}\\
\|\theta_n\|_{L^{\infty}((0,\infty)\times\Omega)}&\leq C,\label{t7glest}\\
\theta&\geq \frac1C,\label{t8glest}
\end{align}
and $C$ and $\widetilde{C}$ depend on $\Omega$, $|\nu|$, $\mu$, $\lambda$, $d$, $\|\bb u_0\|_{\htwo}$, $\|\bb v_0\|_{\hom}$, $\|\theta_0\|_{\hom}$, $\|\theta_0\|_{L^{\infty}(\tn)}$ and $\tilde\theta$ (the bound from below of the initial temperature, see Definition \ref{soldef}). In addition, the constant $\widetilde C$ also depends on $T$.
\end{tw}
\begin{proof}
The standard theory of regularity for elliptic equations (see \cite{Evans, Lady, Amann}) and the choice of an appropriate basis guarantee the sufficient regularity of $\bu_n$ and $\theta_n$ to employ the results from Section \ref{druga}. Inequalities \eqref{t7glest} and \eqref{t8glest} are immediate consequences of Theorem \ref{upbofortemp} and Theorem \ref{downbotemp}.
Theorem \ref{estsol}, and \eqref{t7glest} yield
\begin{align}\label{in1glest}
\nonumber\|\div\bb u_{n,t}\|^2_{\ld}&+\|\curl\bb u_{n,t}\|^2_{\ld}+(\wspnd)\|\nabla\div\bb u_n\|^2_{\ld}+\lambda\|\curl\curl\bb u_n\|^2_{\ld}\\
&+{\|\theta_n\|_{L^{\infty}((0,T)\times\tn)}}^{-1}\|{\nabla\theta_n}\|^2_{\ld}\leq\|\div\bb u_{n,t}\|^2_{\ld}+\|\curl\bb u_{n,t}\|^2_{\ld}\nonumber\\
&+(\wspnd)\|\nabla\div\bb u_n\|^2_{\ld}+\lambda\|\curl\curl\bb u_n\|^2_{\ld}+\left\|\frac{\nabla\theta_n}{\sqrt{\theta_n}}\right\|^2_{\ld}
\leq C.
\end{align}

Thereby, we immediately get \eqref{t2glest}. Thanks to Theorem \ref{prexpr} and \eqref{identity}, we have
\begin{align*}
\|\bb u_n\|_{\htwo}\leq C\|\Delta\bb u_n\|_{\ld}\leq C(\|\nabla\div\bb u_n\|_{\ld}+\|\curl\curl\bb u_n\|_{\ld}).
\end{align*}
Hence, inequality \eqref{in1glest} implies \eqref{t1glest}.

Let us estimate the temperature. We utilize the Gagliardo-Nirenberg inequality and the Young inequality to obtain
\begin{align*}
\|\theta_n\|_{\ld}\leq C(\|\nabla\theta_n\|_{\ld}+\|\theta_n\|_{L^1(\tn)}).
\end{align*}
Let us note that the above, together with \eqref{balance} and \eqref{in1glest}, gives \eqref{t4glest}. Now, we proceed towards \eqref{t3glest}. Let us take $\bb u_{n, tt}$ as $\bs\phi$ in the upper equation in Definition \ref{apdef}. Then, we integrate by parts, and we use the Young inequality, and get
\begin{align*}
\|\bb u_{n,tt}\|^2_{\ld}&=(\wspnd)\io\bb u_{n,tt}\cdot\nabla\div\bb u_n dx-\mu\io\bb u_{n,tt}\cdot\curl\curl\bb u_ndx-\nu\io\nabla\theta_n\cdot\bb u_{n,tt}dx\\
&\leq\frac12\|\bb u_{n,tt}\|^2_{\ld}+C(\|\nabla\div\bb u_n\|_{\ld}^2+\|\curl\curl\bb u_n\|_{\ld}^2+\|\nabla\theta_n\|_{\ld}^2).
\end{align*}
This, thanks to \eqref{in1glest} gives us \eqref{t3glest}.

Since the solution is sufficiently regular, the second equation in Definition \ref{apdef} is satisfied pointwise. Let's square both sides of it and integrate it over $\tn$. It leads us to
\begin{align*}
\io \theta_{n,t}^2dx-2\io \theta_{n,t}\Delta\theta_ndx+\io(\Delta\theta_n)^2dx=\nu^2\io \theta_n^2(\div\bb u_{n,t})^2dx.
\end{align*}
We integrate by parts the second integral on the left-hand side and then integrate it over the interval $[0, T]$. We obtain
\begin{align}\label{intwglestlast}
\|\theta_{n,t}\|_{L^2(0,T;\ld)}^2&+\|\Delta\theta_n\|_{L^2(0,T;\ld)}^2\nonumber\\
&\leq 2\|\nabla\theta_{0,n}\|_{\ld}^2+\nu^2T\|\theta_n\|_{L^{\infty}((0,T)\times\tn)}^2\|\nabla\bb u_{n,t}\|^2_{L^{\infty}(0,T;\ld)}.
\end{align}
The right-hand side is bounded because of \eqref{t2glest} and \eqref{t7glest}. Hence, we get \eqref{t5glest}. Moreover, from \cite[Theorem III.4.3]{Boyer} we have
\begin{align*}
\|\nabla\theta\|_{\ld}+\|\nabla^2\theta\|_{\ld}\leq C\|\Delta\theta\|_{\ld}.
\end{align*}
This together with \eqref{t4glest} and \eqref{intwglestlast} entail \eqref{t6glest}.
\end{proof}

Now, we can prove the existence of solutions.
\begin{tw}
For any initial values satisfying
\begin{align*}
\initcond
\end{align*} there exists a solution to problem \eqref{system} as described in Definition \ref{soldef}. Moreover, we have that there exist constants $C,c>0$ such that
\begin{align}\label{tezexgl}
c\leq\theta\leq C\quad\textrm{for almost all }(t,x)\in [0,\infty)\times\Omega.
\end{align}
\end{tw}
\begin{proof}
By Theorem \ref{twglest} there exists a subsequence of $(\bb u_n,\theta_n)$ (still denoted as $(\bb u_n,\theta_n)$) and
\begin{align*}
\bb u&\in L^{\infty}(0,\infty;\htwo\cap\H)\cap W^{1,\infty}(0,\infty;\H)\cap W^{2,\infty}(0,\infty;\ld),\\
\theta&\in L^{\infty}(0,\infty;\hom)\cap H^1_{\loc}(0,\infty;\ld)\cap L^2_{\loc}(0,\infty;\htwo),
\end{align*}
such that the following weak convergences hold
\begin{align*}
\bb u_n&\stackrel *{\rightharpoonup}\bb u\quad\ \;  \textrm{ in }L^{\infty}(0,\infty;\htwo),\\
\bb u_{n,t}&\stackrel *{\rightharpoonup}\bb u_t\quad\ \textrm{ in }L^{\infty}(0,\infty;\H),\\
\bb u_{n,tt}&\stackrel *{\rightharpoonup}\bb u_{tt}\quad \textrm{ in }L^{\infty}(0,\infty;\ld),\\
\theta_{n}&\stackrel*{\rightharpoonup}\theta\quad\ \ \textrm{ in }L^{\infty}(0,\infty;\hom),\\
\theta_{n,t}&\stackrel{}{\rightharpoonup}\theta_t\quad\ \textrm{ in }L^2_{\loc}(0,\infty;\ld),\\
\theta_{n}&\stackrel{}{\rightharpoonup}\theta\quad\ \ \textrm{ in }L^2_{\loc}(0,\infty;\htwo).
\end{align*}

Due to the Gelfand triple theorem (see, for instance, \cite{Boyer, Evans}), we notice that
\[
\theta\in C([0,\infty);\hom),
\]
\[
\bb u\in C([0,\infty);\H)\;\;\mbox{and}\;\;\bb u_t\in C([0,\infty);\ld).
\]
Moreover, in virtue of the Aubin-Lions lemma (see \cite{Boyer}), we have
\begin{align}\label{zbexgl}
\theta_n\to\theta\quad\textrm{ in }L^{2}_{\loc}(0,\infty;\ld).
\end{align}
Hence, in the standard way, we show that $\bb u$ and $\theta$ satisfy the formulas \eqref{defmom} and \eqref{defheat} in Definition \ref{soldef}. The boundary conditions are implicit in these variational formulations. 

Formula \eqref{zbexgl} implies that there exists the subsequence of $\theta_n$ such that 
\begin{align*}
\theta_n\to\theta \quad \textrm{for almost all }(t,x)\in [0,\infty)\times\Omega.
\end{align*}
Whence and from \eqref{t7glest} and \eqref{t8glest}, we have \eqref{tezexgl}.
\end{proof}

\subsection{Uniqueness of solutions}\label{subuniq}
In the theorem below, we prove the uniqueness of solutions. Let us perceive that we do not assume that the initial values satisfy $\F\leq D$.
\begin{tw}
The solution satisfying \eqref{system} in the regularity class described in Definition \ref{soldef} is unique.
\end{tw}
\begin{proof}
Let $(\bb u_1,\theta_1)$ and $(\bb u_2,\theta_2)$ be solutions of \eqref{system} in the sense of Definition \ref{soldef} with the same initial values. Let us denote $\bb u:=\bb u_1-\bb u_2$ and $\theta:=\theta_1-\theta_2$. First, let us note that the regularity given in Definition \ref{soldef} suffices to apply the proof of Theorem \ref{upbofortemp} to $(\bb u_1,\theta_1)$ and $(\bb u_2,\theta_2)$. Therefore, we can infer that
\begin{equation}\label{lith}
\theta_1,\theta_2\in L^{\infty}(\ot).
\end{equation}

Let us subtract the equations for $(\bb u_2,\theta_2)$ from the equations for $(\bb u_1,\theta_1)$. We obtain
\begin{align}\label{equniq}
\begin{split}
\bb u_{tt}-(\wspnd)\nabla\div\bb u +\mu\curl\curl\bb u&=   -\nu\nabla\theta,\\
\theta_t - \Delta \theta &=- \nu\theta_1 \div{\bb u}_{t}-\nu\theta\div\bb u_{2,t}.
\end{split}
\end{align}
We multiply the upper equation in \eqref{equniq} by $\bb u_t$ and integrate over $\tn$. We get
\begin{align}\label{inuuniq}
\nonumber\ul\ddt\left(\|\bb u_t\|^2_{\ld}+(\wspnd)\|\div\bb u\|_{\ld}^2+\mu\|\curl\bb u\|_{\ld}^2\right)&=-\nu\io\nabla\theta\cdot\bb u_t dx\\\leq|\nu|\|\nabla \theta\|_{\ld}\|\bb u_t\|_{\ld}&
\leq\epsilon\|\nabla\theta\|_{\ld}^2 +C\|\bb u_t\|_{\ld}^2.
\end{align}

Next, we multiply the lower equation in \eqref{equniq} by $\theta$ and integrate over $\tn$. It leads us to
\begin{align}\label{eqthuniq}
\ul\ddt \|\theta\|^2_{\ld}+\|\nabla\theta\|_{\ld}^2=- \nu\io\theta\theta_1 \div{\bb u}_{t}dx-\nu\io\theta^2\div\bb u_{2,t}dx=I_1+I_2.
\end{align}
We shall estimate the integrals on the right-hand side of the above equality.
\begin{align}\label{inmthuniq}
I_1 =\nu \io\theta_1\nabla\theta\cdot\bb u_tdx+\mu\io\theta\nabla\theta_1\cdot\bb u_tdx=I_{11}+I_{12}.
\end{align}
The integral $I_{11}$ is estimated using the Schwarz inequality and the Young inequality. We also utilize \eqref{lith} and obtain
\begin{align}\label{in2thuniq}
I_{11}\leq |\nu|\|\theta_1\|_{L^{\infty}((0,\infty)\times\tn)}\|\nabla\theta\|_{\ld}\|\bb u_t\|_{\ld}\leq \epsilon\|\nabla\theta\|_{\ld}^2+C\|\bb u_t\|_{\ld}^2.
\end{align}
Let us now deal with $I_{12}$. We use the Gagliardo-Nirenberg inequality and the Young inequality to arrive at
\begin{align*}
I_{12}&\leq |\nu| \|\theta\|_{L^4(\tn)}\|\nabla\theta_1\|_{L^4(\tn)}\|\bb u_t\|_{\ld}\\
&\leq C(\|\nabla\theta\|_{\ld}+\|\theta\|_{\ld})(\|\nabla^2\theta_1\|_{\ld}+\|\nabla\theta_1\|_{\ld})\|\bb u_t\|_{\ld}.
\end{align*}
We again utilize the Young inequality. It yields
\begin{align}\label{in1thuniq}
I_{12}\leq \epsilon\|\nabla\theta\|_{\ld}^2+ C(\|\nabla^2\theta_1\|_{\ld}^2\|\bb u_t\|_{\ld}^2+\|\bb u_t\|_{\ld}^2+\|\theta\|_{\ld}^2).
\end{align}

We plug \eqref{in2thuniq} and \eqref{in1thuniq} into \eqref{inmthuniq}. It gives
\begin{align}\label{inmfuniq}
I_1\leq 2\epsilon\|\nabla\theta\|_{\ld}^2+ C(\|\nabla^2\theta_1\|_{\ld}^2\|\bb u_t\|_{\ld}^2+\|\bb u_t\|_{\ld}^2+\|\theta\|_{\ld}^2).
\end{align}
Now, let us deal with $I_2$. We again bound it by the Gagliardo-Nirenberg inequality and the Young inequality
\begin{align}\label{inmf1uniq}
I_2\leq |\nu|\|\div\bb u_{2,t}\|_{L^{\infty}(0,\infty;\ld)}\|\theta\|_{L^4(\tn)}^2\leq \epsilon\|\nabla\theta\|^2_{\ld}+C\|\theta\|_{\ld}^2.
\end{align}
We plug \eqref{inmfuniq} and \eqref{inmf1uniq} into \eqref{eqthuniq}. The obtained inequality we add to \eqref{inuuniq}. It leads us to
\begin{align}\label{razdwatrzy}
\nonumber\ul\ddt&\left(\|\bb u_t\|^2_{\ld}+(\wspnd)\|\div\bb u\|_{\ld}^2+\mu\|\curl\bu\|_{\ld}^2+\|\theta\|^2_{\ld}\right)+\|\nabla\theta\|_{\ld}^2
\\ &\leq4\epsilon\|\nabla\theta\|_{\ld}^2+ C(\|\nabla^2\theta_1\|_{\ld}^2\|\bb u_t\|_{\ld}^2+\|\bb u_t\|_{\ld}^2+\|\theta\|_{\ld}^2).
\end{align}
Taking $\epsilon:=\frac18$ and denoting $y:=\|\bb u_t\|^2_{\ld}+(\wspnd)\|\div\bb u\|_{\ld}^2+\mu\|\curl\bb u\|_{\ld}^2+\|\theta\|^2_{\ld}$, we rewrite \eqref{razdwatrzy} as follows
\begin{align*}
y'\leq C y(1+\|\nabla^2\theta_1\|_{\ld}^2).
\end{align*}
We multiply the above inequality by $\operatorname{exp}\left(-C\int_0^t\|\nabla^2\theta_1\|_{\ld}^2ds-Ct\right)$ for $t\in(0,\infty)$ and obtain
\begin{align*}
\ddt\left(y \operatorname{exp}\left(-C\int_0^t\|\nabla^2\theta_1\|_{\ld}^2ds-Ct\right)\right)\leq 0.
\end{align*}
This entails that $y(t)=0$ for all $t\in [0,\infty)$.
\end{proof}

\section{Time-asymptotics of solutions}\label{czwarta}
In this section, we will investigate the asymptotics of solutions derived from Theorem \ref{existglm}. We will see that the displacement decomposes into two parts, which behave differently when $t\to\infty$. We will apply the Helmholtz decomposition on $\bu$. Then, it will separate into the potential part (also called the rotation-free part) and the divergence-free part. It turns out that the divergence-free part can be considered completely independently of the temperature and the potential part. Indeed, it will be shown that it satisfies the homogeneous wave equation. Consequently, it oscillates indefinitely.

On the other hand, the rotation-free part of the displacement and the temperature are coupled in the same system as \ref{system}. The potential part and $\theta$ behave in the same way as the heated string in \cite{BC2}. Namely, the curl-free part is influenced by the propagation of heat. It means that the whole energy, interlocked with the displacement, is dissipated into heat as $t\to\infty$.

\subsection{Decomposition of solutions}\label{subdec}

In this subsection, we apply the Helmholtz decomposition (Theorem \ref{thdecomp}) to \eqref{system}. Namely, we will demonstrate that the problem divides into two distinct issues.
\begin{tw}
Let $\bb u$ and $\theta$ be solutions of \eqref{system} in the sense of Definition \ref{soldef}. Then, $\bs\gamma:=H\bb u$ solves the following problem
\begin{equation}\label{eqdecnondiv}
\begin{cases}
\bs\gamma_{tt}-\mu\Delta \bs\gamma =0& \quad \text{ in } (0,\infty)\times \tn,\\
\bs\gamma\cdot\bb n=0,\ \N\bs\gamma=0&\quad\text{ in }(0,\infty)\times \partial\Omega,\\
\bs{\gamma}(0,\cdot)=\bs{\gamma}_0, \bs{\gamma}_t(0,\cdot) = \tilde{\bs\gamma}_0.&
\end{cases}
\end{equation}
where $\bs\gamma_0:=H\bb u_0$ and $\tilde{\bs\gamma}_0=H\bb v_0$. On the other hand, $\bs\chi:=\hp \bb u$ and $\theta$ solve
\begin{equation}\label{eqdecnoncurl}
\begin{cases}
\bs\chi_{tt}-(\wspnd)\Delta\bs\chi=-\nu\nabla\theta & \quad \text{ in } (0,\infty)\times \tn,\\
\theta_t - \Delta \theta =- \nu\theta \div{\bs\chi}_t& \quad \text{ in }  (0,\infty)\times\tn,\\
\bs\chi\cdot \bb n=0,\ \nabla\theta\cdot\bb n=0&\text{ on }(0,\infty)\times\partial\Omega,\\
\bs{\chi}(0,\cdot)=\bs{\chi}_0,\ \bs{\chi}_t(0,\cdot) = \tilde{\bs\chi}_0,\ \theta(0,\cdot)=\theta_0>0,&
\end{cases}
\end{equation}
where $\bs\chi_0:=\hp \bb u_0$ and  $\tilde{\bs{\chi}}_0:=\hp \bb v_0$. In addition, the regularity of both $\bs\gamma$ and $\bs\chi$ is equivalent to that of $\bu$ in \eqref{reggl}. Moreover, we have the following decomposition of the energy
\begin{align*}
\ul\io |\bs\gamma_t|^2dx+\frac\mu2\io |\curl \bs\gamma|^2dx=\ul\io |\tilde{\bs\gamma}_0|^2dx+\frac\mu2\io |\curl \gamma_0|^2dx
\end{align*}
and
\begin{align*}
\ul\io |\bs\chi_t|^2dx+\frac\wspnd2\io |\div\bs\chi|^2dx+\io\theta dx=\ul\io |\tilde{\bs\chi}_0|^2dx+\frac\wspnd2\io |\div\bs\chi_0|^2dx+\io\theta_0 dx.
\end{align*}
\end{tw}
\begin{proof}
Let us deal with the wave equation from \eqref{system}. We multiply it by $\bs\gamma_{tt}+\mu\curl\curl\bs\gamma$ and integrate over $\tn$. We obtain
\begin{align*}
\io &|\bg_{tt}+\mu\curl\curl\bg|^2dx+\io(\bs\chi_{tt}-(\wspnd)\nabla\div\bs\chi)\cdot (\bg_{tt}+\mu\curl\curl\bg)dx\\
&=-\nu\io\nabla\theta\cdot(\bg_{tt}+\mu\curl\curl\bg)dx=-\nu\io\nabla\theta\cdot\bg_{tt}-\nu\mu\io\nabla\theta\cdot\curl\curl\bg dx.
\end{align*}
We integrate by parts the integrals on the right-hand side of the above equality. The first integral disappears easily. Let us note that $\curl\bu=\curl\bg$. It implies that $\N\bg=0$ on $\partial\Omega$. Thus, in the second integral, we switch $\curl$ from $\curl\curl\gamma$ onto $\nabla\theta$. This results in the integral also vanishing.

Let us consider the second integral on the left-hand side. We know that $\bs\chi=\hp\bb u$. Thus, there exists a function $\phi$ such that $\nabla\phi=\bs\chi$. Therefore, it yields
\begin{align*}
\io&(\bs\chi_{tt}-(\wspnd)\nabla\div\bs\chi)\cdot (\bg_{tt}-\mu\curl\curl\bg)dx=\io(\nabla\phi_{tt}-(\wspnd)\nabla\Delta\phi)\cdot (\bg_{tt}-\mu\curl\curl\bg)dx.
\end{align*}
We integrate by parts the integrals on the right-hand side of the above equality. It results in its disappearance. Thus, we have that
\begin{align*}
\io |\bg_{tt}-\mu\curl\curl\bg|^2dx=0.
\end{align*}
This, with the formula $\curl\curl\bg=\Delta\bg$, implies that $\bg$ satisfies \eqref{eqdecnondiv}.

We know that $\div \bu=\div\bs\chi$. It shows that the heat equation from \eqref{eqdecnoncurl} is satisfied. The upper is achieved when we subtract \eqref{eqdecnondiv} from \eqref{system}. We also utilize the equality $\nabla\div\bs\chi=\Delta\bs\chi$. Then, the energy decomposition follows in the same way as the proof of Proposition \ref{energy_cons}.
\end{proof}

\begin{rem}
The boundary condition for $\bs\chi$ \eqref{eqdecnoncurl}, i.e. $\bs\chi\cdot\bb n=0$ on $(0,\infty)\times\partial\Omega$, seems to be strange. However, we can rewrite it for $\phi$ such that $\nabla\phi=\bs\chi$. Then, the whole problem presents as follows
\begin{align*}
\begin{cases}
\nabla\phi_{tt}-(2\mu+\lambda)\nabla\Delta\phi =-\nu\nabla\theta & \quad \text{ in } (0,\infty)\times \tn,\\
\theta_t - \Delta \theta =- \nu\theta \Delta\phi_t& \quad \text{ in }  (0,\infty)\times\tn,\\
\nabla\phi\cdot \bb n=0,\ \nabla\theta\cdot\bb n=0&\quad\text{ on }(0,\infty)\times\partial\Omega,\\
\phi(0,\cdot)=\psi,\ \phi_t(0,\cdot) = \tilde{\psi},\ \theta(0,\cdot)=\theta_0>0,&
\end{cases}
\end{align*}
where $\psi,\tilde\psi$ we take from $\nabla\psi=\bs\chi_0$ and $\nabla\tilde\psi=\tilde{\bs\chi}_0$. Therefore, we see that $\phi$ satisfies the Neumann boundary condition.
\end{rem}

\subsection{Displacement and temperature at infinity}\label{subas}

From the previous section, we know that the displacement decomposes into two different vector fields
\begin{align*}
\bu=\bs\gamma+\bs\chi.
\end{align*}
It occurs that $\bs\gamma$ and $\bs\chi$ behave differently when $t\to\infty$. We start by noticing that the non-divergence part of $\bu$ (i.e., $\bg$) oscillates unless it is zero (cf.\cite{Bies}). 
\begin{prop}
Let $\bu$ and $\theta$ be solutions of \eqref{system}. Then,
\begin{align}\label{connondiv}
\bg:=H\bu\to 0\textrm{ in } \ld, \textrm{ when } t\to\infty
\end{align}
if and only if $\bg$ is equal to $0$.
\end{prop}
\begin{proof}
It is obvious that \eqref{connondiv} is satisfied when $\bg=0$.

Now, let us assume that \eqref{connondiv} is valid. Let us remind that $\{\bs\phi_k\}$ is an orthonormal basis of $\ld$ consisting of eigenfunctions of $-\Delta$ with the boundary condition as in \eqref{probeig}. It is, in addition, orthogonal in $\H$. Let us also denote the sequence of eigenvalues of $-\Delta$ on $\tn$ as $\{\xi_k\}$. Hence, we can write $\bg=\sum_{k=1}^{\infty} d_k \bs\phi_k$. Equation \eqref{eqdecnondiv} implies that $d_k$ is a solution of the following problem
\begin{align*}
\begin{cases}
d_k''+\mu\xi_kd_k=0&\textrm{on } [0,\infty),\\
d_k(0)=(\bg_0,\bs\phi_k)_{\ld},\ d_i'(0)=(\tilde\bg_0,\bs\phi_i)_{\ld}.
\end{cases}
\end{align*}
Condition \eqref{connondiv} implies that $d_k=(\bg,\bs\phi_k)_{\ld}\to 0$, when $t\to\infty$. Because $\xi_k>0$, it could happen only if $d_k=0$. It means that $\bg=0$.
\end{proof}

Let us now study the asymptotics of the potential part of $\bu$ (i.e., $\bs\chi$) and the temperature. Our considerations are similar to those in \cite{BC2, Biest}. Nevertheless, certain details need to be converted. We begin with the following proposition regarding the stability of solutions.
\begin{prop}\label{twcontdep}
Let us take two triples of the initial values $\bu_{0,1},\bu_{0,2}\in \H\cap\htwo, \bb v_{0,1},\bb v_{0,2}\in\H$ and $\theta_{0,1},\theta_{0,2}\in\hom\cap L^{\infty}(\tn)$. Moreover, we  require that there exists $\tilde\theta_1>0$ and $\tilde\theta_2>0$ such that $\tilde\theta_1\leq\theta_{0,1}$ and $\tilde\theta_2\leq\theta_2$.
Let us also assume that $(\bu_1,\theta_1)$ is a solution of \eqref{system} starting from $\bu_{0,1},\bb v_{0,1},\theta_{0,1}$ and $(\bu_2,\theta_2)$ is a
solution starting from $\bu_{0,2},\bb v_{0,2},\theta_{0,2}$. Both of the solutions are considered in the sense of Definition \ref{soldef}. Then, for all $t\in(0,\infty)$ there exists the constant $C>0$  such that
\begin{align*}
&\|\bu_1(t,\cdot)-\bu_2(t,\cdot)\|_{\H}+\|\bu_{1,t}(t,\cdot)-\bu_{2,t}(t,\cdot)\|_{\ld}+\|\theta_1(t,\cdot)-\theta_2(t,\cdot)\|_{\ld}\\
&\quad\leq C\left(\|\bu_{0,1}-\bu_{0,2}\|_{\H}+\|\bb v_{0,1}-\bb v_{0,2}\|_{\ld}+\|\theta_{0,1}-\theta_{0,2}\|_{\ld}\right)
\end{align*}
and $C=C(t,d,\mu,\lambda,|\nu|, \|\nabla \bu_{1,t}\|_{L^{\infty}(0,\infty;\ld)},\|\theta_2\|_{L^{\infty}((0,\infty)\times\tn)},\|\nabla \theta_2\|_{L^{\infty}(0,\infty;\ld)}, \|\nabla^2\theta_2\|_{L^2(0,t;\ld)})$.
\end{prop}
\begin{proof}
Let us denote $\bu:=\bu_1-\bu_2$ and $\theta:=\theta_1-\theta_2$. We subtract the equations for $(\bu_2,\theta_2)$ from the equations for $(\bu_1,\theta_1)$ and
obtain
\begin{align}\label{uklpom}
\begin{cases}
\bu_{tt}-(\wspnd)\nabla\div\bu+\mu\curl\curl\bu=-\nu\nabla\theta,\\
\theta_t-\Delta\theta=-\nu\theta \div \bu_{1,t}-\mu \theta_2\div \bu_{t}.
\end{cases}
\end{align}
We multiply by $\bu_t$ the first of the above equations and arrive at
\begin{align}\label{glrow1}
\ul\ddt\left(\io \bu_t^2dx+(\wspnd)\io (\div \bu)^2dx+\mu\io |\curl \bu|^2dx \right)&\leq-\nu\io \bu_t\cdot\nabla\theta dx\nonumber\\
&\leq C\|\bu_t\|_{\ld}^2+\epsilon\|\nabla \theta\|^2_{\ld}.
\end{align}
Whereas, multiplying the second equation in \eqref{uklpom} by $\theta$, we obtain
\begin{align}\label{glrow2}
\nonumber\ul\ddt&\io\theta^2 dx+\io|\nabla\theta|^2 dx=-\nu\io\theta^2\div \bu_{1,t}dx-\nu\io\theta\theta_2\div \bu_{t}dx\\
&=-\nu\io\theta^2\div \bu_{1,t}dx+\nu\io\theta\nabla\theta_2\cdot \bu_{t}dx+\mu\io\theta_2\nabla\theta\cdot \bu_{t}dx=I_1+I_2+I_3.
\end{align}

We bound the integral $I_1$ using the Gagliardo-Nirenberg inequality and then the Young inequality. We get
\begin{align}\label{inconde1}
I_1\leq |\nu|\|\div \bu_{1,t}\|_{L^{\infty}(0,T;\ld)}\|\theta\|_{L^4(\tn)}^2\leq \epsilon\|\nabla\theta\|^2_{\ld}+C\|\theta\|_{\ld}^2.
\end{align}
The integral $I_{3}$ is estimated using the Schwarz inequality and the Young inequality. We also utilize that $\theta_2\in L^{\infty}(\ot)$ for the class of solutions given in Definition \ref{soldef}.
\begin{align}\label{inconde2}
I_{3}\leq |\nu|\|\theta_2\|_{L^{\infty}((0,T)\times\tn)}\|\nabla\theta\|_{\ld}\|\bu_t\|_{\ld}\leq \epsilon\|\nabla\theta\|_{\ld}^2+C\|\bu_t\|_{\ld}^2.
\end{align}

Let us proceed with $I_{2}$. We use the Gagliardo-Nirenberg inequality and the Young inequality.
\begin{align*}
I_{2}&\leq |\nu| \|\theta\|_{L^4(\tn)}\|\nabla\theta_2\|_{L^4(\tn)}\|\bu_t\|_{\ld}\\
&\leq C(\|\nabla\theta\|_{\ld}+\|\theta\|_{\ld})(\|\nabla^2\theta_2\|_{\ld}+\|\nabla\theta_2\|_{\ld})\|\bu_t\|_{\ld}.
\end{align*}
Next, after the application of the Young inequality, we get
\begin{align}\label{inconde3}
I_{2}\leq \epsilon\|\nabla\theta\|_{\ld}^2+ C(\|\nabla^2\theta_2\|_{\ld}^2\|\bu_t\|_{\ld}^2+\|\bu_t\|_{\ld}^2+\|\theta\|_{\ld}^2).
\end{align}

We plug \eqref{inconde1}, \eqref{inconde2} and \eqref{inconde3} into \eqref{glrow2}. The obtained inequality we add to \eqref{glrow1}. It leads us to
\begin{align*}
\ul\ddt&\left(\|\bu_t\|^2_{\ld}+(\wspnd)\|\div \bu\|_{\ld}^2+\mu\|\curl \bu\|_{\ld}^2+\|\theta\|^2_{\ld}\right)+\|\nabla\theta\|_{\ld}^2
\\ &\leq4\epsilon\|\nabla\theta\|_{\ld}^2+ C(\|\nabla^2\theta_2\|_{\ld}^2\|\bu_t\|_{\ld}^2+\|\bu_t\|_{\ld}^2+\|\theta\|_{\ld}^2).
\end{align*}
We take $\epsilon:=\frac18$.

Let us denote $y:=\|\bu_t\|^2_{\ld}+(\wspnd)\|\div \bu\|_{\ld}^2+\mu\|\curl \bu\|_{\ld}^2+\|\theta\|^2_{\ld}$. Then, we rewrite the above inequality as follows
\begin{align*}
y'\leq C y(1+\|\nabla^2\theta_2\|_{\ld}^2).
\end{align*}
We multiply the above inequality by $\operatorname{exp}\left(-C\int_0^s\|\nabla^2\theta_2\|_{\ld}^2dz-Cs\right)$ for $s\in(0,t]$. We obtain
\begin{align*}
\frac{d}{ds}\left(y \operatorname{exp}\left(-C\int_0^s\|\nabla^2\theta_2\|_{\ld}^2dz-Cs\right)\right)\leq 0.
\end{align*}
We integrate over $[0,t]$ and see that
$$
y(t)\leq y(0)\operatorname{exp}\left(C\int_0^t\|\nabla^2\theta_2\|_{\ld}^2ds+Ct\right).
$$
It gives the thesis.
\end{proof}

We denote
\begin{align*}
\M:=\left\{\eta\in \hom\cap L^{\infty}(\Omega)\colon \operatorname*{ess\, inf}_{x\in\tn}\eta(x)>0\right\}.
\end{align*}
Let us take $t\geq 0$. We define an operator
\begin{align*}
S(t):\left(\H\cap\htwo\right)\times\H\times\M\to \left(\htwo\cap\H\right)\times\H\times\M
\end{align*}
by formula
\begin{align*}
S(t)(\bu_0,\bb v_0,\theta_0)=(\bu(t),\bu_t(t),\theta(t)),
\end{align*}
where $\bu$ and $\theta$ are solutions of \eqref{system} with initial values $\bu_0, \bb v_0$ and $\theta_0$. Thanks to the uniqueness of the solutions of the system, we obtain
\begin{align*}
S(t+h)=S(h)\circ S(t)
\end{align*}
for all $t,h\in [0,\infty)$.

Now, we are ready to prove the main result concerning the long-time behavior of solutions. Let us remind that $\bs\chi=\hp \bu$, $\bs\chi_0=\hp \bu_0$, $\bs{\tilde\chi}_0=\hp\bb v_0$.
\begin{proof}[The proof of Theorem \ref{twassm}]
Let us take a sequence $t_n\in [0,\infty)$ such that $t_n\to\infty$. We consider the sequences $\bs\chi(t_n,\cdot)$, $\bs\chi_t(t_n,\cdot)$ and
$\theta(t_n,\cdot)$. Due to the very Definition \ref{soldef}, we know that
\begin{align}\label{cont}
\begin{split}
\bs\chi\in C([0,\infty);&\H),\ \chi_t \in C([0,\infty);\ld),\\
\theta& \in C([0,\infty);\hom).
\end{split}
\end{align}
Hence, we see that the sequence $\theta(t_n,\cdot)$ is bounded in $\hom$. Nevertheless, about the potential part, we know only that
\begin{align*}
\bs\chi\in L^{\infty}(0,\infty;\htwo),\quad \bs\chi_t\in L^{\infty}(0,\infty,\H).
\end{align*}

Therefore, there exists a measurable set $\T\subset [0,\infty)$ such that $\left|[0,\infty)\setminus\T\right|=0$ and $\bs\chi|_{\T}$, $\bs\chi_t|_{\T}$ are bounded
functions with values in $\htwo$ and $\H$, respectively.
However, if we show
\begin{align*}
\lim_{\substack{t\to\infty\\t\in\T}}\bs\chi(t,\cdot)=0\textrm{ in }\H
\end{align*}
and
\begin{align*}
\lim_{\substack{t\to\infty\\t\in\T}}\bs\chi_t(t,\cdot)=0\textrm{ in }\ld,
\end{align*}
then \eqref{cont} will imply that
\begin{align*}
\lim_{t\to\infty}\bs\chi(t,\cdot)=0\textrm{ in }\H,
\end{align*}
and
\begin{align*}
\lim_{t\to\infty}\bs\chi_t(t,\cdot)=0\textrm{ in }\ld.
\end{align*}
The details are left to the reader.

Hence, we consider $\bs\chi(t_n,\cdot)$ and $\bs\chi_t(t_n,\cdot)$ as bounded in $\htwo$ and $\H$, respectively. We also know that $\theta(t_n,\cdot)$ is
bounded in $\hom$. Thus, we have a subsequence of $t_n$ (still denoted as $t_n$) such that
\begin{align*}
\bs\chi(t_n,\cdot)&\rightharpoonup \tilde{\bs\chi}\textrm{ in }\htwo,\nonumber\\
\bs\chi_t(t_n,\cdot)&\rightharpoonup\tilde{\bb v}\textrm{ in }\H,\\
\theta(t_n,\cdot)&\rightharpoonup\theta_{\infty}\textrm{ in }\hom,\nonumber
\end{align*}
where $\tilde{\bs\chi}\in\htwo$, $\tilde{\bb v}\in\H$ and $\theta_{\infty}\in\hom$. We know that $\bs\chi(t_n,\cdot)\in\hp(\htwo)$ and $\bs\chi_t(t_n,\cdot)\in\hp(\H)$ for all $n$. Because $\hp(\htwo)$ and $\hp(\H)$ are convex and closed subsets of $\htwo$ and  $\H$, respectively, we obtain that $\tilde{\bs\chi}\in\hp(\htwo)$ and $\tilde{\bb v}\in\hp(\H)$. By compact embeddings, we know that
\begin{align*}
\begin{split}
\bs\chi(t_n,\cdot)&\to \tilde{\bs\chi}\textrm{ in }\H,\\
\bs\chi_t(t_n,\cdot)&\to\tilde{\bb v}\textrm{ in }\ld,\\
\theta(t_n,\cdot)&\to\theta_{\infty}\textrm{ in }\ld.
\end{split}
\end{align*}
By virtue of Theorem \ref{estsol} and assumption \eqref{assyassu}, we have that 
\begin{align*}
\F(\bs\chi(t_n,\cdot),\bs\chi_t(t_n,\cdot),\theta(t_n,\cdot))\leq \F(\bu(t_n,\cdot),\bu_t(t_n,\cdot),\theta(t_n,\cdot))\leq D
\end{align*}
for all $n\in\n$. This implies that
\begin{align*}
\F(\tilde{\bs\chi},\tilde{\bb v},\theta_{\infty})\leq D.
\end{align*}
Therefore, Theorem \ref{existglm} ensures that there exist solutions $\bar{\bs\chi}$ and
$\bar\theta$ of \eqref{system} with the initial values $\tilde{\bs\chi}$, $\tilde{\bb v}$ and $\theta_{\infty}$.

Let us denote $\bar{\bs\gamma}:=H\bar{\bs\chi}$. Then, the function $\bar{\bs\gamma}$ is a solution of the following problem
\begin{align*}
\begin{cases}
\bar{\bs\gamma}_{tt}-\mu\Delta\bar{\bs\gamma}=0 &\text{ in }\ot,\\
\bar{\bs\gamma}\cdot\bb n=0,\ \N\bar{\bs\gamma}=0 &\text{ on } (0,\infty)\times\partial\Omega,\\
\bar{\bs\gamma}(0,\cdot)=0,\ \bar{\bs\gamma}_t(0,\cdot)=0.
\end{cases}
\end{align*}
This entails that
\begin{align*}
\ul\ddt\left(\io|\bar{\bs\gamma}_t|^2dx+\mu\io|\curl\bar{\bs\gamma}|^2dx\right)=0
\end{align*}
Thus, we have that $\bar{\bs\gamma}=0$ and then $\bar{\bs\chi}\in\hp(\htwo)$. Consequently, there exists a function $\phi$ such that $\bar{\bs\chi}=\nabla\phi$ and $\io\phi dx=0$.

Subsequently, we take $h>0$ and consider the sequences $\bs\chi(t_n+h,\cdot)$, $\bs\chi_t(t_n+h,\cdot)$ and $\theta(t_n+h,\cdot)$. Similarly, as above, there exists
$\hat{\bs\chi}\in\htwo\cap\H$, $\hat{\bb v}\in\H$ and $\hat\theta\in\hom$ such that
\begin{align*}
\bs\chi(t_n+h,\cdot)&\to \hat{\bs\chi}\textrm{ in }\H,\\
\bs\chi_t(t_n+h,\cdot)&\to\hat{\bb v}\textrm{ in }\ld,\\
\theta(t_n+h,\cdot)&\to\hat{\theta}\textrm{ in }\ld.
\end{align*}
If it is necessary, we take the subsequence of $t_n$ in the above convergences. By Proposition \ref{twcontdep}, we obtain
\begin{align*}
(\hat{\bs\chi},\hat{\bb
v},\hat\theta)&=\lim_{n\to\infty}(\bs\chi(t_n+h,\cdot),\bs\chi(t_n+h,\cdot),\theta(t_n+h,\cdot))=\lim_{n\to\infty}S(h)(\bs\chi(t_n,\cdot),\bs\chi_t(t_n,\cdot),\theta(t_n,\cdot))\\
&=S(h)(\tilde{\bs\chi},\tilde{\bb v},\theta_{\infty})=(\bar{\bs\chi}(h,\cdot),\bar{\bs\chi}_t(h,\cdot),\bar\theta(h,\cdot)).
\end{align*}

Whereas, by Proposition \ref{ogrzprop}, we know that a function
\begin{align*}
t\mapsto\io\log\theta(t,x)dx
\end{align*}
is non-decreasing. In view of \eqref{tezexglm}, it is bounded. Hence, it is convergent when time goes to $\infty$. Thus, sequences $\io\log\theta(t_n,x)dx$ and $\io\log\theta(t_n+h,x)dx$ must be convergent to the same number. Moreover,
\begin{align}\label{piec}
\nonumber\io\log\theta_{\infty} dx&=\lim_{n\to\infty}\io\log\theta(t_n,x)dx=\lim_{n\to\infty}\io\log\theta(t_n+h,x)dx\\
&=\io\log\bar\theta(h,x)dx.
\end{align}
The equalities for limits can be concluded from the inequality
\begin{align}\label{inhel}
|\log y_1-\log y_2|\leq\frac1c|y_1-y_2| \text{ for all }y_1,y_2\in[c,\infty),
\end{align}
where $c>0$. Inequality \eqref{tezexglm} gives us that $\theta$, $\theta_{\infty}$, and $\bar\theta$ are bounded away from $0$. As the constant $c$, in \eqref{inhel}, we take the bound from below for $\theta$, $\theta_{\infty}$, and $\bar\theta$.  Then, inequality \eqref{inhel} implies
\begin{align*}
\left|\io\log\theta(t_n,x)dx-\io\log\theta_{\infty} dx\right|\leq \frac1c\|\theta(t_n,\cdot)-\theta_{\infty}\|_{L^1(\tn)}\leq C\|\theta(t_n,\cdot)-\theta_{\infty}\|_{\ld}
\end{align*}
and similarly
\begin{align*}
\left|\io\log\theta(t_n+h,x)dx-\io\log\bar\theta(h,x) dx\right|\leq C\|\theta(t_n+h,\cdot)-\bar\theta(h,\cdot)\|_{\ld}
\end{align*}
Thus, we see that \eqref{piec} is got from limits $\theta(t_n,\cdot)\to\theta_{\infty}$ and $\theta(t_n+h,\cdot)\to\bar\theta(h,\cdot)$ in $\ld$.

Because $h>0$ was arbitrary, we notice that \eqref{piec} states that
\begin{align*}
t\mapsto\io\log\bar\theta(t,x)dx
\end{align*}
is constant. Thus, we obtain that the integral of entropy is constant at the $\omega$-limit set.  Afterwards, \eqref{sepri} infers
\begin{align*}
0=\ddt\io\log\bar\theta dx=\io\frac{|\nabla \bar\theta|^2}{\bar\theta^2}dx.
\end{align*}
Hence,
\[
\nabla \bar{\theta}=0 \quad\text{ for all }\quad(t,x)\in\ot.
\]
It entails that, for each fixed $t>0$, $\bar{\theta}(t,\cdot)$ is constant in space. Next, again we make use of \eqref{sepri}.
\begin{align*}
0=\ddt\io\log\bar\theta dx=\ddt\log\bar\theta =\frac{\bar\theta_t}{\bar\theta}.
\end{align*}
It means that $\bar\theta$ is also constant in time. Therefore, $\bar{\theta}$ is constant in space and time.

So far, we inferred that the $\omega$-limit set of a solution of \eqref{system} consists only of constant functions $\theta_{\infty}$. At the end of the proof, we will see that there is only one limit temperature. Next, we identify all potential limits of the curl-free part of the displacement. Since $\bar{\theta}$ is constant, we have
\begin{align}\label{szesc}
\begin{cases}
\bar{\bs\chi}_{tt}-(\wspnd)\Delta\bar{\bs\chi}=0& \text{ in }\ot,\\
0=\nu\bar \theta\div\bar{\bs\chi}_{t}&\text{ in }\ot,\\
\bar{\bs\chi}\cdot\bb n=0,&\text{ on }(0,\infty)\times\partial\Omega
\end{cases}
\end{align}

The second equation and the boundary condition in the above system show $\phi_t$ (let us remind that $\nabla\phi=\bar{\bs\chi}$) is a solution of the following boundary problem  
\begin{align*}
\begin{cases}
\Delta\phi_t=0 & \text{ in }\ot,\\
\nabla\phi_t\cdot\bb n=0&\text{ on }(0,\infty)\times\partial\Omega,\\
\io\phi_tdx=0.
\end{cases}
\end{align*}
Therefore, the function $\phi_t=0$ on $\ot$. Hence, $\bar{\bs\chi}_t=\nabla\phi_t=0$ on $\ot$ and we obtain $\bar{\bs\chi}$ is constant in time. This, together with the first equation in \eqref{szesc}, states that $\Delta\bar{\bs\chi}=0$. Because $\bar{\bs\chi}=\nabla\phi$, we know that $\curl\bar{\bs\chi}=0$. It results that $\bar{\bs\chi}$ solves
\begin{align*}
\begin{cases}
\Delta\bar{\bs\chi}=0 & \text{ in }\ot,\\
\bar{\bs\chi}\cdot\bb n=0,\ \N\bar{\bs\chi}=0&\text{ on }(0,\infty)\times\partial\Omega.
\end{cases}
\end{align*}
Theorem \ref{prexpr} yields that $\bar{\bs\chi}=0$ on $\ot$. Thus, we have that $\tilde{\bs\chi}=0$, $\tilde{\bb v}=0$, and $\theta_{\infty}$ is constant.
Finally, the conservation of energy (Proposition \ref{energy_cons}) implies that
\begin{align*}
\frac12&\io |\tilde{\bs\chi}_0|^2dx+\frac12\io |\nabla\bs\chi_0|^2dx+\io\theta_0 dx\\
&=\lim_{n\to\infty}\left(\frac12\io |\bs\chi_t(t_n,x)|^2dx+\frac12\io |\nabla\bs\chi(t_n,x)|^2dx+\io\theta(t_n,x) dx\right)\\
&=\io\theta_{\infty} dx=\theta_{\infty}|\tn|.
\end{align*}
At the end, let us notice that the time-limit of $(\bs\chi,\bs\chi_t,\theta)$ does not depend on the taken sequence $\{t_n\}$, $t_n\to\infty$. Indeed, for an arbitrary sequence $\{t_n\}$, we obtain that 
\begin{align*}
\bs\chi(t_n,\cdot)&\to 0\textrm{ in }\H,\\
\bs\chi_t(t_n,\cdot)&\to 0\textrm{ in }\ld,\\
\theta(t_n,\cdot)&\to\theta_{\infty}\textrm{ in }\ld,
\end{align*}
where $\theta_{\infty}$ is constant in time and space and is the same as in the formula in the thesis. 
\end{proof}

\subsection*{Acknowledgments}
P. M. Bies would like to thank Tomasz Cie\'slak for reading the initial draft of the manuscript and for helpful discussions about the paper.
\\

\noindent{\bf Data availability} Data sharing not applicable to this article as no datasets were generated or analysed during
the current study.\\
\\
\noindent{\bf Declarations}

\noindent{\bf Conflict of interest} The author declares that he has no conflict of interest.
\bibliographystyle{ACM}
\bibliography{boundary}	
\end{document}